\documentclass[12pt,usenames,dvipsnames]{amsart}
\usepackage{amsmath, amssymb, amsthm, array, bm, calrsfs, comment, enumitem, graphicx, hyperref, latexsym, makecell, mathpazo, tikz-cd}
\usepackage[margin=1in]{geometry}
\usepackage[cmtip, all]{xy}

\newtheorem{theorem}{Theorem}[section]
\newtheorem{lemma}[theorem]{Lemma}
\newtheorem{proposition}[theorem]{Proposition}
\newtheorem{corollary}[theorem]{Corollary}

\theoremstyle{definition}
\newtheorem{example}[theorem]{Example}

\theoremstyle{remark}
\newtheorem{remark}[theorem]{Remark}

\numberwithin{equation}{section}
\numberwithin{table}{section}
\numberwithin{figure}{section}

\newcommand{\GG}{{\mathbb G}}	
\newcommand{\PP}{\mathbb{P}}         
		
\newcommand{\RR} {{\mathbb R}}		
\newcommand{\ZZ} {{\mathbb Z}}	
\newcommand{\GGm}{{\mathbb G}_m}	

\renewcommand{\AA}{\ensuremath{\mathbb{A}}} 
\newcommand{\KK}{\ensuremath{\mathbb{K}}}

\DeclareMathOperator{\Gr}{Gr}

\DeclareMathOperator{\Span}{span}
\DeclareMathOperator{\id}{id} 
\DeclareMathOperator{\Pf}{Pf} 
\DeclareMathOperator{\A}{A} 
\DeclareMathOperator{\D}{D}

\DeclareMathOperator{\val}{val}

\DeclareMathOperator{\Stab}{Stab}
\DeclareMathOperator{\rank}{rank}

\DeclareMathOperator{\Spec}{{Spec}}
\DeclareMathOperator{\Proj}{{Proj}}
\DeclareMathOperator{\init}{in}
\DeclareMathOperator{\Trop}{Trop}

\DeclareMathOperator{\conv}{conv}

\DeclareMathOperator{\St}{Star}
\DeclareMathOperator{\GL}{GL}

\DeclareMathOperator{\Dr}{Dr}

\DeclareMathOperator{\PGL}{PGL}
\DeclareMathOperator{\Orth}{O}
\DeclareMathOperator{\sch}{sch}

\providecommand{\Sn}[1]{\ensuremath{\mathfrak{S}_{#1}}}
\DeclareMathOperator{\face}{face}

\DeclareMathOperator{\sgn}{sgn}
\DeclareMathOperator{\sg}{sg}

\DeclareMathOperator{\Sec}{S}

\DeclareMathOperator{\Top}{top}
\DeclareMathOperator{\semigp}{semigp}
\DeclareMathOperator{\even}{ev}
\DeclareMathOperator{\odd}{odd}
\DeclareMathOperator{\rowsp}{rowsp}
\DeclareMathOperator{\codim}{codim}

\newcommand{\bS}{\mathbb{S}}

\DeclareMathOperator{\TS}{T\bS}

\newcommand{\GrU}[2]{\Gr_0(#1,#2)}
\newcommand{\GrC}[2]{\Gr(#1,#2)}

\newcommand{\SF}[1]{\Sigma_{\Sec}(#1)}

\newcommand{\ETrop}{\mathfrak{Trop}}

\newcommand{\supp}{\operatorname{supp}}

\newcommand{\set}[2]{\left\{ #1 \, : \, #2  \right\}}
\newcommand{\chow}[2]{#1 /\!\!\!/\! #2}

\newcommand{\kk}{\mathbf{k}}

\newcommand{\bn}{[n]}

\newcommand{\calP}{\cal}

\def\pB{{\calP B}}

\newcommand{\cal}{\mathcal}

\def\cB{{\cal B}}
\def\cC{{\cal C}}

\def\cO{{\cal O}}
\def\cP{{\cal P}}

\def\cU{{\cal U}}

\newcommand{\bracket}[2]{\langle #1,#2 \rangle}

\usepackage{tikz}
\usepackage{mathtools,tikz-cd}  

\tikzset{%
	add/.style args={#1 and #2}{to path={%
			($(\tikztostart)!-#1!(\tikztotarget)$)--($(\tikztotarget)!-#2!(\tikztostart)$)%
			\tikztonodes}}
}

\newcommand{\gfan}{\texttt{gfan}}

\newcommand{\zspan}[1]{\ensuremath{\ZZ\!\cdot\!#1}}
\newcommand{\rspan}[1]{\ensuremath{\RR\!\cdot\!#1}}

\newcommand{\kSch}{\mathop{\mbox{$\kk$-$\sch$}}}

\title{Initial degenerations of spinor varieties}

\author{Daniel Corey}

\date{\today}

\address{Technische Universit\"at Berlin, Stra{\ss}e des 17 Juni 135, Berlin 10623, Germany.}

\email{corey@math.tu-berlin.de}

\begin{document}

\begin{abstract}
We construct closed immersions from initial degenerations of the spinor variety $\bS_n$ to inverse limits of strata associated to even $\Delta$-matroids. As an application, we prove that these initial degenerations are smooth and irreducible for $n\leq 5$ and identify the log canonical model of the Chow quotient of $\bS_5$ by the action of the diagonal torus of $\GL(5)$. 

	\smallskip
	\noindent \textbf{Keywords}: Spinor varieties, initial degenerations, Delta-matroids, Chow quotient.
	
	\smallskip
	\noindent \textbf{Mathematics Subject Classification}: 14T05 (primary), 15A66, 52B40 (secondary).
\end{abstract}

\maketitle

\section{Introduction}
The pure spinor variety, denoted $\bS_n^{\pm}$, parameterizes totally isotropic subspaces of a $2n$-dimensional vector space $V$ equipped with a symmetric, nondegenerate quadratic form of Witt index $n$. From the representation-theoretic viewpoint, $\bS_{n}^{\pm}$ is the Lie-type $\D$ analog of the Grassmannian. The Wick embedding, similar to the Pl\"ucker embedding, realizes $\bS_{n}^{\pm}$ as a closed subvariety of $\PP(\wedge^*E)$, where $E$ is some fixed totally isotropic subspace of $V$. The exterior algebra of $E$ splits as a direct sum of its even and odd parts, i.e.,  $\wedge^*E = \wedge^{\even} E \oplus \wedge^{\odd} E$, and $\bS_n^{\pm}$ has two (projectively equivalent) irreducible components $\bS_n$ and $\bS_n^{-}$, one contained in $\PP(\wedge^{\even} E)$ and the other in $\PP(\wedge^{\odd} E)$. Our convention is that $\bS_n\subset \PP(\wedge^{\epsilon(n)}E)$ where $\epsilon(n)$ is the parity of $n$.

We investigate the initial degenerations of $\bS_{n}^{\circ}$---the open dense cell of $\bS_{n}$ given by the nonvanishing of all  Wick coordinates with $\epsilon(n)$ parity---via their relation to isotropical linear spaces in the sense of Rinc\'on \cite{Rincon}. 
This work is an extension to the Lie-type $\D$ setting of \cite{Corey17} where the author investigates the initial degenerations of the Grassmannian via their relation to tropical linear spaces. Specifically, a vector in the tropical Grassmannian records an initial degeneration and a tropical linear space, and the latter records a diagram of matroid strata parameterized by a regular matroidal subdivision of the hypersimplex. The main result in \textit{loc. cit.} produces a closed immersion from the initial degeneration into the inverse limit of this diagram. This result has a number of applications to $\GrU{d}{n}$---the open dense cell of $\GrC{d}{n}$ given by the nonvanishing of all Pl\"ucker coordinates---especially $\GrU{3}{7}$.  Specifically, the initial degenerations of $\GrU{3}{7}$ are smooth and irreducible, the tropicalization of $\GrU{3}{7}$ with respect to the Pl\"ucker embedding is faithful, and the Chow quotient of $\GrC{3}{7}$ by the diagonal torus of $\PGL(7)$ is the log canonical compactification of the moduli space of 7 lines in $\PP^2$ in linear general position, partially resolving a conjecture of Hacking, Keel, and Tevelev \cite[Conjecture~1.6]{KeelTevelev2006}. As Grassmannians and matroids are Lie-type $\A$ partial flag varieties and Coxeter matroids, respectively, we expect that many of the results in \cite{Corey17} extend to a broader class of partial flag varieties. In this paper, we confirm this expectation in the Lie-type $\D$ setting.

There are numerous formulations of Lie-type $\D$ Coxeter matroids \cite{BorovikGelfandWhite, Bouchet87, Bouchet97}; we use the language of $\Delta$-matroids.  Given $w$ in $\TS_n^{\circ} := \Trop(\bS_{n}^{\circ})$, we may form an initial degeneration $\init_w\bS_{n}^{\circ}$ and an isotropical linear space; the latter encodes a regular subdivision $\Delta_{w}$ of the $\Delta$-hypersimplex $\Delta(n)$
into $\Delta$-matroid polytopes \cite{Rincon}.  
The collection of totally isotropic subspaces realizing a $\Delta$-matroid $M$ forms a locally-closed subscheme  $\bS_{M}\subset \bS_n^{\pm}$.  If $Q$ is a $\Delta$-matroid polytope and $Q'$ is a face, then $Q'$ is also a $\Delta$-matroid polytope (by the Gelfand-Serganova theorem \cite[Theorem~6.3.1]{BorovikGelfandWhite}), and there is a morphism $\bS_{M_{Q}} \to \bS_{M_{Q'}}$, where, e.g., $M_{Q}$ denotes the $\Delta$-matroid of $Q$. These morphisms determine a diagram of type $\Delta_w$, hence we may form the inverse limit $\varprojlim_{Q\in \Delta_w} \bS_{M_Q}$. 

\begin{theorem}
	\label{thm:closedImmersionIntro}
	There is a closed immersion $\init_{w}\bS_n^{\circ} \hookrightarrow \varprojlim_{Q\in \Delta_w} \bS_{M_Q}$. 
\end{theorem}

We give a geometric interpretation of this closed immersion in \S \ref{sec:symmetricMatroids}; here is a summary. To define a morphism $\init_w\bS_n^{\circ} \to \varprojlim_{Q\in \Delta_w} \bS_{M_Q}$, we must define a morphism $\init_w\bS_n^{\circ} \to \bS_{M_Q}$ for each $Q\in \Delta_{w}$ (and these morphisms must be compatible with the face morphisms).
Let $\KK = \kk(\!(t^{\RR})\!)$,  equip it with its usual $t$-adic valuation, and let $x$ be a $\kk$-point of $\init_w\bS_{n}^{\circ}$. There is a $\KK$-point $q$ of $\bS_n^{\circ}$ whose exploded tropicalization is $x$ \cite{Payne}. Let $F_{q}\subset \KK^{2n}$ be the totally isotropic subspace with Wick vector $q$, and $F_q^{\circ} = F_{q} \cap (\KK^*)^{2n}$.  If $v$ is in the cell of $\Trop(F_{q}^{\circ})$ dual to $Q\in \Delta_w$, then  $\overline{\init_{v}F_q^{\circ}}$ is a totally isotropic subspace of $\kk^n$ that realizes $M_{Q}$. The map $\init_{w}\bS_{n}^{\circ} \to \bS_{M_Q}$ takes $x$ to (the Wick vector of) $\overline{\init_{v}F_q^{\circ}}$.

Working with the full poset $\Delta_w$ is impractical due to its size. Nevertheless, the inverse limit  $\varprojlim_{Q\in \Delta_w} \bS_{M_Q}$ depends only on the adjacency graph $\Gamma_w$ of $\Delta_w$. The vertices of $\Gamma_w$ correspond to the codimension 0 cells, and two vertices are connected by an edge whenever their corresponding cells intersect along a facet. In \S \ref{sec:TechniquesLimitsStrata}, we show how this encodes a diagram of $\Delta$-matroid strata, form the inverse limit $\varprojlim_{\Gamma_w} \bS_{M_Q}$, and deduce
\begin{equation*}
\varprojlim_{Q\in \Delta_w} \bS_{M_Q} \cong \varprojlim_{\Gamma_w} \bS_{M_Q}.
\end{equation*}
The analogous statement in the Grassmannian setting is due to Cueto \cite[Appendix~C]{Corey17}. We develop various techniques for evaluating these inverse limits, which we use, together with Theorem \ref{thm:closedImmersionIntro}, to prove the following theorem.

\begin{theorem}
	\label{thm:introSchon}
	The initial degenerations of $\bS_{n}^{\circ}$ are smooth and irreducible for $n\leq 5$. 
\end{theorem}

We conclude by studying the birational geometry of $\chow{\bS_n}{H}$, the Chow quotient of $\bS_n$ by the maximal torus $H$ of the orthogonal group, for $n\leq 5$.  
There is an ordered basis 
of $V$ so that the matrix of the quadratic form is 
\begin{equation}
\label{eq:Q}
Q = \begin{bmatrix}
0 & I \\
I & 0
\end{bmatrix}   
\end{equation}
The \textit{orthogonal group} $\Orth(2n)$ (with respect to the symmetric form $Q$) and its maximal torus $H$ are 
\begin{equation*}
 \Orth(2n) = \{ X \in \GL(2n) \, : \, X^{T} Q X = Q \} \hspace{10pt} \text{and} \hspace{10pt} H = \set{\begin{bmatrix}
 	h^{-1} & 0 \\
 	0 & h
 	\end{bmatrix}}{h\in \GL(n) \text{ is diagonal}}.
\end{equation*}
Any totally isotropic subspace is the row-span of a full-rank $n\times 2n$  matrix of the form $[A|B]$ where $AB^T$ is skew-symmetric \cite[Lemma~3.4.1]{BorovikGelfandWhite}. The action $H\curvearrowright \bS_{n}$ is given by right multiplication:
\begin{equation}
\label{eq:Haction}
[ A | B ] \begin{bmatrix}
h^{-1} & 0 \\
0 & h
\end{bmatrix} = [ Ah^{-1} | Bh ].
\end{equation}

Let us describe how to realize $\chow{\bS_n}{H}$ as a compactification of  $\bS_{n}^{\circ}/H$ inside of a toric variety.  We may realize $H$ as a subtorus of the dense torus $T$ of $\PP(\wedge^{\epsilon(n)}E)$, and the scaling action of $T$ on $\PP(\wedge^{\epsilon(n)}E)$ restricts to the action of $H$ on $\bS_{n}$ from above. The Wick embedding induces an embedding of Chow quotients $\chow{\bS_{n}}{H} \hookrightarrow \chow{\PP(\wedge^{\epsilon(n)}E)}{H}$.  By \cite{KapranovSturmfelsZelevinsky}, $\chow{\PP(\wedge^{\epsilon(n)}E)}{H}$ is the toric variety of the secondary fan of $\Delta(n)/L_{\RR}$, where $L$ is the cocharacter lattice of $H$ and $L_{\RR} = L\otimes_{\ZZ} \RR$. The locus of matroidal subdivisions is the $\Delta$-Dressian $\Dr(n)$, and we denote by $\Sigma_n$ the secondary fan of $\Dr(n)$.  The Chow quotient $\chow{\bS_{n}}{H}$ is the closure of $\bS_n^{\circ}/H$ in the toric variety $X(\Sigma_n/L_{\RR})$.

As a consequence of Theorem \ref{thm:introSchon}, for $n\leq 5$,  $\bS_{n}^{\circ}/H$ is sch\"on in the sense of Tevelev \cite{Tevelev}, and $\chow{\bS_{n}}{H}$ is a sch\"on compactification of $\bS_n^{\circ}/H$. By general results of Hacking, Keel, and Tevelev in \cite{HackingKeelTevelev2009}, determining the log canonical model of $\chow{\bS_{n}}{H}$ amounts to finding a suitable fan structure of $\Trop(\bS_{n}^{\circ})$ (which equals $\Dr(n)$ by \cite[Theorem~4.5]{Rincon}). We show that $\chow{\bS_{n}}{H}$ is log canonical for $n = 4$, see Proposition \ref{prop:n4LogCanonical}.  For $n=5$, there is a coarser fan $\Sigma_{5}'$ supported on $\Dr(5)$; denote by $Y$ the closure of $\bS_5^{\circ}/H$ in $X(\Sigma_5'/L_{\RR})$. 
\begin{theorem}
	\label{thm:logcanonical}
	The Chow quotient $\chow{\bS_{5}}{H}$ is smooth and has a simple normal crossings boundary. The log canonical model of $\chow{\bS_{5}}{H} $ is $Y$ and the refinement $\Sigma_{5} \to \Sigma_{5}'$ induces a log crepant resolution $\chow{\bS_{5}}{H} \to Y$. 
\end{theorem}

\subsection*{Computations and data availability} 
The software package \gfan~ \cite{gfan} was used to compute the tropicalization of $\bS_{5}^{\circ}$ in \S \ref{sec:chow}, and the $\Delta$-matroid subdivisions in Appendix \ref{appendix:subdivisionsn5} were computed using \texttt{polymake}~ \cite{polymake:2000}. Also, \texttt{sage} is used to verify Lemma \ref{lem:YInvariantSubspace}.  No computation takes more than a few seconds on a standard desktop computer. The code may be found at the following website.

\begin{center}
\url{https://github.com/dcorey2814/tropicalSpinorVarieties}
\end{center}

\subsection*{Conventions}
Let $\kk$ be an algebraically closed field of characteristic 0. 
Let $[n]= \{0,1,\ldots,n-1\}$,  $[n]^{*} = \{0^*,1^*,\ldots, (n-1)^*\}$, and $J=[n]\sqcup [n^*]$. Given a set $X$, let ${X \choose d}$ the set of all $d$-element subsets of $X$.  We  use juxtaposition of elements to denote small subsets of $J$, e.g., $ij=\{i,j\}$, $ijk=\{i,j,k\}$, etc. So, given $\lambda\subset \bn$,  we write $\lambda \cup i = \lambda \cup \{i\}$, $\lambda \setminus i = \lambda \setminus \{i\}$, and $\lambda \Delta i = \lambda \Delta \{i\}$, etc.

\section{The pure spinor variety}
In this section, we recall the construction of the spinor variety following the conventions in \cite{Manivel, Rincon}. Let $V=\kk^{2n}$ and $Q$ the quadratic form from \eqref{eq:Q}. An $n$-dimensional vector subspace $F\subset V$ is \textit{totally isotropic} if $x^TQy = 0$ for all $x,y\in F$. The \textit{pure spinor variety} is 
\begin{equation*}
\bS_n^{\pm} = \set{F\subset V}{F \text{ is an } n \text{ dimensional totally isotropic subspace of } V}. 
\end{equation*}
This space has two irreducible components $\bS_n$ and $\bS_{n}^{-}$, which may be distinguished in the following way.   Given $\mu \subset \bn$, set 
\begin{equation*}
\overline{\mu} = \mu \cup \set{i^*\in [n]^{*}}{i\notin \mu}.
\end{equation*}
Note that $|\overline{\mu}|=n$ for any $\mu$. Given any $d$-dimensional subspace $F$ of a $m$-dimensional vector space and $\lambda \in {[m]\choose d}$, let $p_{\lambda}(F)$ be the $\lambda$-th Pl\"ucker coordinate of $F$. The varieties $\bS_n$ and $\bS_{n}^{-}$ are
\begin{align*}
\bS_n = \set{F\in \bS_n^{\pm}}{p_{\overline{\mu}} =0 \text{ if }  n - |\mu|  \text{ is odd} }, \hspace{10pt} 
\bS_n^{-} = \set{F\in \bS_n^{\pm}}{p_{\overline{\mu}} =0 \text{ if }  n - |\mu|  \text{ is even} }.
\end{align*}
Moreover, this characterization describes $\bS_n$ as a closed subvariety of the Grassmannian $\GrC{n}{2n} \subset \PP(\wedge^dV)$. Nonetheless, $\bS_{n}$ embeds into a smaller projective space, as we now describe. 

Denote by $E(n)$ the set of subsets $\lambda$ of $\bn$ such that $n-|\lambda|$ is even.  The \textit{Wick embedding} realizes $\bS_n$ as a closed subvariety of $ \PP(\kk^{E(n)})$; we recall the definition of this map on the open dense affine chart
\begin{equation}
\label{eq:affineChart}
\cU  = \set{F\in\bS_n}{p_{\bn}(F)\neq 0}.
\end{equation}
the Wick embedding looks similar on other affine charts obtained by replacing $[n]$ by $\overline{\mu}$ for any $\mu\subset [n]$ (such a subset is called a \textit{transversal}, see \S \ref{sec:symmetricMatroids}). 

In terms of matrices, if $F\in \cU$, then $F$ is the row span of the $n\times 2n$ matrix $\left[I_n|X \right]$ where $X$ is skew-symmetric.  Given $\lambda\subset \bn$, let $X[\lambda]$ denote the skew-symmetric matrix obtained from the rows and columns of $X$ indexed by $\lambda$. Let $q_{\lambda}(F) = \Pf(X[\bn\setminus \lambda])$ where $\Pf$ denotes the Pfaffian of a skew-symmetric matrix. Note that $q_{\lambda}(F) = 0$ if $n-|\lambda|$ is odd.  The Wick embedding restricted to $\cU$ is
\begin{equation*}
\cU \to \PP(\kk^{E(n)}) \hspace{20pt} F\mapsto [\, q_{\lambda}(F) \, | \, \lambda \in E(n)].
\end{equation*}
Thus, $\bS_n$ is the closure of $\cU$ in $\PP(\kk^{E(n)})$. 

Denote by $\kk[q_{\lambda}]$ the polynomial ring $\kk[q_{\lambda} \, | \, \lambda\in E(n)]$. The homogeneous ideal $I_n\subset \kk[q_{\lambda}]$ of $\bS_{n}$ is generated by the quadrics
\begin{equation}
\label{eq:quadricGens}
P(\mu,\nu) = \sum_{i\in \nu\setminus\mu} (-1)^{\sgn(i;\mu,\nu)} q_{\mu\cup i}q_{\nu\setminus i} \; + \; \sum_{j\in \mu\setminus\nu} (-1)^{\sgn(j;\nu,\mu)} q_{\mu\setminus i}q_{\nu\cup j}.
\end{equation}
where $n-|\mu|$ and $n-|\nu|$ are odd and $|\mu\Delta\nu| \geq 4$; otherwise $P(\mu,\nu) = 0$. In particular, a monomial $q_{\eta}q_{\lambda}$ cannot appear in both sums in the above expression (this is be important in the proofs of Propositions \ref{prop:mapPlucker} and \ref{prop:mapInitToStratum}).  Here,  $\sgn(i;\mu,\nu)$ equals $(-1)^{\ell}$ where $\ell$ is the number of elements $j\in \nu$ with $i<j$ plus the number of elements $j'\in \mu$ such that $i>j'$. 

We end this section by considering the first few examples. When $n=1,2,$ or $3$, $\bS_n=\PP(\kk^{E(n)})$. The first interesting case is $n=4$, where
\begin{equation*}
I_4 = \langle q_{\emptyset} q_{0123} - q_{01}q_{23} + q_{02}q_{13} - q_{03}q_{12} \rangle \subset \kk[q_{\emptyset}, q_{01}, q_{02}, q_{12}, q_{03}, q_{13}, q_{23},q_{0123}].
\end{equation*}
When $n=5$, $I_5$ is the ideal of 
\begin{equation*}
\kk[q_{0},q_{1},q_{2},q_{3},q_{4}, q_{012}, q_{013}, q_{023}, q_{123}, q_{014}, q_{024}, q_{124}, q_{034}, q_{134}, q_{234}, q_{01234}]
\end{equation*}
generated by the quadrics 
\begin{align*}
\begin{array}{ll}
q_{0} q_{123} - q_{1} q_{023} + q_{2} q_{013} - q_{3} q_{012}, & 
q_{0} q_{124} - q_{1} q_{024} + q_{2} q_{014} - q_{4} q_{012}, \\
q_{0} q_{134} - q_{1} q_{034} + q_{3} q_{014} - q_{4} q_{013}, & 
q_{0} q_{234} - q_{2} q_{034} + q_{3} q_{024} - q_{4} q_{023}, \\
q_{1} q_{234} - q_{2} q_{134} + q_{3} q_{124} - q_{4} q_{123}, & 
q_{0} q_{01234} - q_{012} q_{034} + q_{013} q_{024} - q_{023} q_{014}, \\
q_{1} q_{01234} - q_{012} q_{134} + q_{013} q_{124} - q_{014} q_{123}, &
q_{2} q_{01234} - q_{012} q_{234} + q_{023} q_{124} - q_{024} q_{123}, \\
q_{3} q_{01234} - q_{013} q_{234} + q_{023} q_{134} - q_{034} q_{123}, &
q_{4} q_{01234} - q_{014} q_{234} + q_{024} q_{134} - q_{034} q_{124}.
\end{array}
\end{align*}

\section{Polytopes and strata of $\Delta$-matroids}

A \textit{$\Delta$-matroid} on $[n]$ is characterized by a nonempty collection $\cB(M)$ of subsets of $[n]$ satisfying the \textit{symmetric exchange axiom}, i.e., for all $\mu,\nu\in \cB(M)$ and $i\in \mu\Delta\nu$ there is a $j\in \mu\Delta\nu$, such that $\mu\Delta ij \in \cB(M)$. Elements of $\cB(M)$ are called \textit{bases} of $M$.  A $\Delta$-matroid is \textit{even} if $|\mu\Delta\nu|$ is even for all $\mu,\nu \in \cB(M)$.

\subsection{$\Delta$-matroid polytopes} 
Let $e_{1}, \ldots, e_{n}$ be the standard basis of $\RR^n$, and given $\lambda = \{\lambda_1, \ldots, \lambda_k\}$, let $e_{\lambda} = e_{\lambda_1} + \cdots + e_{\lambda_k}$. Write $\bracket{u}{v}$ for the standard inner product of the vectors $u,v\in \RR^n$. The polytope of a $\Delta$-matroid $M$ is 
\begin{equation*}
Q_M =  \conv \set{ e_{\lambda} }{\lambda \in \cB(M)}.
\end{equation*}
An (even) $\Delta$-matroid polytope is any polytope of the form $Q_M$ for some (even) $\Delta$-matroid $M$.

The faces of a polytope $Q\subset \RR^n$ are all of the form 
\begin{equation*}
\face_uQ = \set{x\in Q}{\langle u,x\rangle \leq \langle u,y  \rangle \text{ for all } y \in Q} \hspace{20pt} \text{ where } u\in \RR^n.
\end{equation*}
Because the status of being an even $\Delta$-matroid polytope is determined by its edges (a consequence of the Gelfand-Serganova theorem, see \cite[Theorem~3.5]{Rincon} for a version adapted to $\Delta$-matroids), a face of an even $\Delta$-matroid polytope is an even $\Delta$-matroid polytope. Denote by $M_u$ the even $\Delta$-matroid such that $Q_{M_u} = \face_u Q_M$; this is the analog of the initial matroid defined in \cite{ArdilaKlivans}. The bases of $M_u$ are
\begin{equation}
\label{eq:faceBases}
\cB(M_u) = \set{\lambda \in  \cB(M)}{\bracket{u}{e_{\lambda}} \leq \bracket{u}{e_{\lambda'}} \text{ for all } \lambda'\in \cB(M)}.
\end{equation}

\subsection{$\Delta$-matroid strata}
Given a totally isotropic subspace $F$ of $V$, its $\Delta$-matroid, denoted by $M(F)$, is defined by
\begin{equation}
\label{eq:matroidF}
\cB(M(F)) = \set{ \lambda\subset \bn}{ p_{\overline{\lambda}}(F) \neq 0 }.
\end{equation}
This $\Delta$-matroid is necessarily even. A $\Delta$-matroid is \textit{realizable} if it is the $\Delta$-matroid of a totally isotropic subspace of some $V$.  Let $\bS_M\subset \bS_n^{\pm}$ be the locus of isotropic subspaces that realize $M$. We describe $\bS_M$ as a scheme in the following way. Define
\begin{itemize}[noitemsep]
	\item[-] $B_{M} = \kk[q_{\lambda} \, : \, \lambda \in \cB(M)]$;
	\item[-] $I_M = (\langle q_{\lambda} \, : \, \lambda \in E(n) \setminus \cB(M) \rangle  + I_n)\cap B_M$;
	\item[-] $S_{M}$ the multiplicative semigroup generated by $q_{\lambda}$ for $\lambda \in \cB(M)$. 
\end{itemize}
Then $\bS_M = T_{M} \cap \Proj(B_{M}/I_M)$ where $T_M$ is the dense torus of $\Proj(B_M)$. Frequently it is easier to work with $\Spec(R_M) \cong \bS_{M} \times \GGm$, where $R_{M} = S_{M}^{-1}B_{M}/I_M$.  The ideal $I_M$ is generated by the quadrics 
\begin{equation}
\label{eq:quadricGenerators}
P_M(\mu,\nu) = \sum_{i\in \nu\setminus\mu} (-1)^{\sgn(i;\mu,\nu)} q_{\mu\cup i}q_{\nu\setminus i} + \sum_{j\in \mu\setminus\nu} (-1)^{\sgn(j;\nu,\mu)} q_{\mu\setminus i}q_{\nu\cup j}
\end{equation}
such that $\mu\cup i$, $\mu\setminus j$, $\nu\setminus i$, $\nu\cup j$ are all bases of $M$; compare this with Equation \eqref{eq:quadricGens}.

Similar to the Grassmannian case, the face inclusion $Q_{M_u} \subset Q_M$ induces a morphism of strata $\bS_{M} \to \bS_{M_u}$, as we see in the following proposition.

\begin{proposition}
	\label{prop:mapPlucker} 
	Suppose $M$ is a $\kk$-realizable even $\Delta$-matroid and $u\in \RR^n$. Then the inclusion $B_{M_u}\subset B_{M}$ induces a morphism of strata $\varphi_{M,M_u}:\bS_{M} \to \bS_{M_u}$.	
	These morphisms satisfy $\varphi_{M,M} = \id_{\bS_M}$ and $\varphi_{M,(M_u)_v} = \varphi_{M_u,(M_u)_v} \circ \varphi_{M,M_u}$. 
\end{proposition}

\begin{proof}
	It suffices to show that the extension of $I_{M_u}\subset B_{M_u}$ to $B_M$  is contained in $I_M$. Using the quadric generators from \eqref{eq:quadricGenerators}, we must show that  $P_{M_u}(\mu,\nu) = 0$ or $P_{M_u}(\mu,\nu) = P_{M}(\mu,\nu)$. Suppose $P_{M_u}(\mu,\nu) \neq 0$.  Then  there is an $i_0\in \nu\setminus \mu$ such that $\mu\cup i_0$ and $\nu\setminus i_0$ are bases of $M_u$, or there is a $j_0\in \mu\setminus \nu$ such that $\mu\setminus j_0$ and $\nu\cup j_0$ are bases of $M_u$.  The two situations are symmetric, so we only consider the first one.  By \eqref{eq:faceBases}, we have $\bracket{u}{e_{\mu\cup i_0}} = \bracket{u}{e_{\nu\setminus i_0}} \leq \bracket{u}{e_{\lambda}}$ for all $\lambda\in \cB(M)$. We must show that, for all $i\in \nu\setminus \mu$, (resp. $j\in \mu \setminus \nu$), $\mu\cup i$ and $\nu\setminus i$ are bases of $M$ if and only if they are bases of $M_u$ (resp. $\mu\setminus j$ and $\nu\cup j$ are bases of $M$ if and only if they are bases of $M_u$).  Because $\cB(M_u) \subset \cB(M)$, we need only show the ``only if'' directions. Suppose $\mu\cup i$ and $\nu\setminus i$ are bases of $M$. Then  
	\begin{equation*}
	\bracket{u}{e_{\mu}} + \bracket{u}{e_{i_0}} \leq \bracket{u}{e_{\mu}} + \bracket{u}{e_{i}},  \hspace{10pt}\text{ and } \hspace{10pt} \bracket{u}{e_{\nu}} - \bracket{u}{e_{i_0}} \leq \bracket{u}{e_{\nu}} - \bracket{u}{e_{i}},
	\end{equation*}
	so $\bracket{u}{e_{i}}=\bracket{u}{e_{i_0}}$. Therefore $\bracket{u}{e_{\mu\cup i_0}} = \bracket{u}{e_{\mu\cup i}}$ and $\bracket{u}{e_{\nu\setminus i_0}} = \bracket{u}{e_{\nu\setminus i}}$, that is, $\mu\cup i$ and $\nu\setminus i$ are bases of $M_u$. Now suppose $\mu\setminus j$ and $\nu\cup j$ are bases of $M$. Then  
	\begin{equation*}
	\bracket{u}{e_{\mu}} + \bracket{u}{e_{i_0}} \leq \bracket{u}{e_{\mu}} -  \bracket{u}{e_{j}} ,  \hspace{10pt}\text{ and} \hspace{10pt} \bracket{u}{e_{\nu}} -  \bracket{u}{e_{i_0}}  \leq \bracket{u}{e_{\nu}} +  \bracket{u}{e_{j}},
	\end{equation*}
	so $ \bracket{u}{e_{j}} =-\bracket{u}{e_{i_0}}$. Therefore $\bracket{u}{e_{\mu\cup i_0}} = \bracket{u}{e_{\mu\setminus j}}$ and $\bracket{u}{e_{\nu\setminus i_0}} = \bracket{u}{e_{\nu\cup j}}$, that is,  $\mu\setminus j$ and $\nu\cup j$ are bases of $M_u$, as required. 
	That these morphisms satisfy the requisite functorial properties is clear. 
\end{proof}

\subsection{Symmetries of $\Delta$-matroids}
\label{sec:symmetries}
Given a set $X$, denote by $\Sn{X}$ the symmetric group on $X$; if $X=\bn$, we write $\Sn{n} := \Sn{\bn}$.  For $\tau\in \Sn{n}$, let $s_{\tau}\in \Sn{E(n)}$ be the permutation 
\begin{equation*}
s_{\tau}:E(n) \to E(n) \hspace{20pt} \lambda \mapsto \set{\tau(i)}{i\in \lambda}.
\end{equation*}
If $\mu\subset \bn$ has an even number of elements, we have a permutation of $E(n)$ given by
\begin{equation*}
t_{\mu}:E(n) \to E(n) \hspace{20pt} \lambda \mapsto \mu \Delta \lambda  
\end{equation*}  
Let $G_n = \{  t_{\mu} \, : \, \mu \subset \bn,\, |\mu| \text{ even}   \} \subset \Sn{E(n)}$; this is a group isomorphic to $(\Sn{2})^{n-1}$. Following  \cite[VI.4.8]{Bourbaki02}, the type-$D_n$ Weyl group, denoted by $W(D_n)$, is
\begin{equation*}
W(D_n) = \langle s_{\tau}, \, t_{\mu} \, : \, \tau\in \Sn{n},\, \mu \subset \bn,\, |\mu| \text{ even}   \rangle \leq \Sn{E(n)}.
\end{equation*}
The subgroup $G_n$ of $W(D_n)$ is normal since 
\begin{equation*}
\label{eq:conjugationst}
s_{\tau} t_{\mu} s_{\tau}^{-1} = t_{s_{\tau}(\mu) }. 
\end{equation*}	
From this, we see that $W(D_{n}) \cong \Sn{n} \ltimes (\Sn{2})^{n-1}$ and $|W(D_n)| = n!2^{n-1}$. 

The action of $W(D_n)$ on $E(n)$ induces an action on the set of subsets of $E(n)$. Thus $W(D_n)$ acts on the set of even $\Delta$-matroids via its action on the bases sets. For example, given an even $\Delta$-matroid $M$ and an even-sized subset $\lambda\subset [n]$, the \textit{twist} of $M$ by $\lambda$ is the even $\Delta$-matroid with bases $\cB(M\Delta\lambda) = \set{\mu\Delta\lambda}{\mu \in \cB(M)}$;  this is just the action of $t_{\lambda}\in W(D_n)$ on $\cB(M)\in \cP(E(n))$.

\section{Limits of Spinor strata}

\subsection{The tropical spinor variety}
\label{sec:tropSpinor}

We recall tropicalization of embedded varieties from the initial degenerations viewpoint; for details see \cite[\S\S 2.4-5]{MaclaganSturmfels2015}.     Let $\KK$ be a field with valuation $\val$, uniformizing parameter $t$, and residue field $\kk$  (e.g., $\KK = \kk(\!(t^{\RR})\!)$). The valuation plays a significant role in \S \ref{sec:symmetricMatroids}; the reader interested only in $\bS_n$ may assume that the valuation is trivial. Let  $\PP^a = \Proj(\KK[x_0,\ldots,x_a])$, $T\subset \PP^a$ the dense torus, and $N$ the cocharacter lattice of $T$. Via the coordinates on $\PP^a$, we identify $N$ with $\ZZ^{a+1}/\zspan{(1,\ldots,1)}$. Let $X\subset \PP^a$ be a closed irreducible subvariety not contained in any coordinate hyperplane, $I\subset \KK[x_0,\ldots,x_a]$ its homogeneous ideal, and $X^{\circ}=X\cap T$. If $z = (z_{0},\ldots, z_{a})\in \ZZ^{a+1}$, we write $x^{z} = x^{z_0}\cdots x^{z_a}$. Given $w\in N_{\RR}:= N\otimes_{\ZZ} \RR$, the \textit{$w$-initial form} of $f\in \KK[x_0,\ldots,x_a]$ is
\begin{equation*}
\init_w f = \sum_{\bracket{w}{u} + \val(c_{u}) = W } \overline{t^{-\val(c_u)} c_{u}} x^{u}  \in \kk[x_{0},\ldots,x_{a}]   
\end{equation*}
where $f = \sum c_{u} x^{u}$ and $W = \min\{\bracket{w}{u} + \val(c_{u}) \, : \, c_{u}\neq 0 \}$. The \textit{initial ideal} of $I$ is the homogeneous ideal  
\begin{equation*}
\init_wI = \set{\init_w f}{f\in I} \subset \kk[x_{0}, \ldots, x_{a}].
\end{equation*}
The \textit{tropicalization} of $X^{\circ}$ is 
\begin{equation*}
\Trop X^{\circ} = \set{w\in N_{\RR}}{\init_wI_0 \neq \langle 1 \rangle}.
\end{equation*}
This set is the support of a rational polyhedral complex $\Sigma_{\Gr}$, called the \textit{Gr\"obner complex}, where $w,w'$ belong to the same relatively open cell of $\Sigma_{\Gr}$ if and only if $\init_wI = \init_{w'}I$. If $\KK=\kk$, then $\Sigma_{\Gr}$ is a polyhedral fan. The \textit{initial degeneration} of $X^{\circ}$ (with respect to $w\in \Trop X^{\circ}$) is the scheme
\begin{equation*}
\init_wX^{\circ} = T \cap \Proj(\kk[x_0,\ldots,x_a] / \init_wI).
\end{equation*}
Frequently, it is easier to work with
\begin{equation*}
\Spec(\kk[x_0^{\pm},\ldots,x_a^{\pm}] / \init_wI \cdot \kk[x_0^{\pm},\ldots,x_a^{\pm}]). 
\end{equation*} 
which is isomorphic to $\init_wX^{\circ} \times \GGm$. We remark that $\init_wX^{\circ}$ depends on the cone of $\Sigma_{\Gr}$ that contains $w$ in its relative interior, but there may exist $w,w'$ belonging to different locally closed cones such that $\init_w X^{\circ} = \init_{w'} X^{\circ}$. When this happens, $\init_wI$ and  $\init_{w'}I$ differ by primary components contained in $\langle x_0,\ldots, x_a \rangle$. 

We now specialize to the spinor variety.  Let $n\geq 3$ and $\bS_n^{\circ}$ denote the intersection of $\bS_n$ with the dense torus in $\PP(\kk^{E(n)})$.  Note that $\bS_{n}^{\circ} = \bS_M$ where $M$ is the uniform even $\Delta$-matroid, i.e., $\pB(M) = E(n)$. Viewing $\bS_M$ as a closed subvariety of the algebraic torus $\GG_m^{\pB(M)} / \GG_m$, we may form
\begin{equation*}
	\TS_{M} := \Trop(\bS_M),
	\hspace{20pt} \text{and} \hspace{20pt}
	\TS_{n}^{\circ} := \Trop(\bS_n^{\circ})
\end{equation*}
Label the coordinates of $N_{\RR} \cong \RR^{E(n)}/\rspan{(1,\ldots,1)}$ by $f_{\lambda}$ for $\lambda\in E(n)$. The tropicalization $\TS_n^{\circ}$ is the support of a ${n \choose 2}$-dimensional polyhedral fan. 

Consider the $L \cong \ZZ^{n}$-grading $\deg_L(q_{\lambda}) = \sum_{i \in \lambda} e_i$. The following proposition is clear from the quadric generators of $I_n$ \eqref{eq:quadricGens}.
\begin{proposition}
\label{prop:InHomogeneous}
	The ideal $I_{n}$ is homogeneous with respect to the $L$-grading. 
\end{proposition}
\noindent Therefore, $\TS_n^{\circ}$ has a $n$-dimensional lineality space $L_{\RR}\subset N_{\RR}$ where
\begin{equation}
\label{eq:Lineality}
L = \left\langle \sum_{\lambda \ni i} f_{\lambda}, \; \sum_{\lambda \not\ni i} f_{\lambda} \, : \, 0 \leq i \leq  n-1\right\rangle \subset N.
\end{equation} 
Note that $L$ is a \textit{saturated} subgroup of $N$, so $N/L$ is torsion-free. Let us describe $\TS_n^{\circ}$ for small values of $n$. To simplify our description of $\TS_{n}^{\circ}$, we list the $W(D_n)$-orbits of the cones, where  $W(D_n)$  acts on $\TS_n^{\circ}$ by 
\begin{equation*}
s_{\tau}\cdot f_{\lambda} = f_{s_{\tau}(\lambda)}, \hspace{20pt} t_{\mu}\cdot f_{\lambda} = f_{t_{\mu}(\lambda)}.
\end{equation*}

When $n=3$, $\TS_3^{\circ} = L_{\RR} = N_{\RR}$. Next, consider the case $n=4$. The saturated subgroup $L\subset N$ is spanned by
{\small
\begin{align*}
\begin{array}{cccc}
f_{01}+f_{02}+f_{03}+f_{0123}, & f_{01} + f_{12} + f_{13}+f_{0123}, & f_{02} + f_{12} + f_{23}+f_{0123}, & f_{03}+f_{13} + f_{23}+f_{0123}, \\
f_{\emptyset}+f_{01}+f_{02}+f_{12}, & f_{\emptyset}+f_{01} + f_{03} + f_{13}, & f_{\emptyset}+f_{02} + f_{03} + f_{23}, & f_{\emptyset}+f_{12}+f_{13}+f_{23}.
\end{array}
\end{align*} 
}
There are 4 rays, which have primitive vectors (modulo $L_{\RR}$)
\begin{equation}
\label{eq:rays4}
r_0 = f_{\emptyset} + f_{0123},\, r_1 = f_{01} + f_{23},\, r_2 = f_{02} + f_{13},\, r_3 = f_{03} + f_{12},
\end{equation}
and 6 maximal cones, one corresponding to each pair of rays. Up to $W(D_4)$-symmetry, there is only one ray and one maximal cone. The space $\TS_5^{\circ}$ is described in \S \ref{sec:chow}. In summary, it has a $5$ dimensional lineality space, and $f$-vector (resp. $W(D_5)$-symmetric $f$-vector):
\begin{equation*}
f(\TS_5^{\circ}) = (1, 36, 280, 960, 1540, 912), \hspace{20pt} \text{resp.} \hspace{20pt} f(\TS_5^{\circ} \mod W(D_5)) = (1, 2, 3, 5, 5, 4).
\end{equation*}

\subsection{Subdivisions of $\Delta$-matroid polytopes}

Given a polytope $Q\subset \RR^n$ with vertices $v_{0},\ldots,v_k$ and $w\in \RR^{k+1}$, the \textit{lifted polytope} is
\begin{equation*} 
Q^{w} = \conv\set{(v_i, w_{i})}{0\leq i\leq k} \subset \RR^{n} \times \RR.
\end{equation*}
Any lower face of $Q^{w}$ is of the form $\face_{\mathbf{u}}Q^w$ where $\mathbf{u} = (u,1)\in \RR^{n}\times \RR$.  The lower faces of $Q^{w}$ project onto $Q$, forming a polyhedral complex  whose support is $Q$. This is called the \textit{regular subdivision} of $Q$ induced by $w$. The \textit{adjacency graph} of this subdivision is the graph with vertex $v_{Q_i}$ for each maximal cell $Q_i$ and an edge between $v_{Q_i}$ and $v_{Q_j}$ whenever $Q_i$ and $Q_j$ share a common facet. 
The \textit{secondary fan}  $\SF{Q}$ of $Q$ is the complete fan in $\RR^{k+1}$ where $w$ and $w'$ belong to the relative interior of the same cone if and only if they induce the same regular subdivision on $Q$ \cite[\S 7.C]{GelfandKapranovZelevinsky}.

Given an even $\Delta$-matroid $M$ and $w\in \RR^{\cB(M)}$, we write $\Delta_{M,w}$ for the regular subdivision of $Q_M$ induced by $w$.  This subdivision is \textit{matroidal}, or $\Delta_{M,w}$ is a \textit{matroid subdivision},  if  each $Q\in\Delta_{M,w}$ is an even $\Delta$-matroid polytope. The \textit{$\Delta$-Dressian} of $M$ is the subfan of $\SF{Q_M}$ defined by
\begin{equation*}
\Dr_M = \set{w\in \RR^{\cB(M)}}{\Delta_{M,w} \text{ is matroidal} }.
\end{equation*}
When $\cB(M)=E(n)$, we write $Q_M=\Delta(n)$, $\Dr_M = \Dr(n)$, and $\Sigma_{S}(Q_{M}) = \Sigma_n$. 

Suppose $\Delta_{M,w}$ is matroidal. Given $u \in \RR^n$, let  $M_{u}^{w}$ be the matroid of the polytope  in $\Delta_{M,w}$ determined by $\face_{\mathbf{u}}(Q_M^{w})$. The bases of $M_{u}^{w}$ are
\begin{equation}
\label{eq:basesCell}
\cB(M_{u}^{w}) = \set{\lambda\in \cB(M)}{ w_{\lambda} + \bracket{u}{e_{\lambda}}  \leq w_{\lambda'} + \bracket{u}{e_{\lambda'}} \text{ for all } \lambda' \in \cB(M) }.
\end{equation}
If $w\in \TS_M$, then $\Delta_{M,w}$ is matroidal \cite[Theorem~5.4]{Rincon}. In fact, for $n\leq 5$, there is an equality of fans $\TS_n^{\circ} = \Dr(n)$,  and for $n\geq 7$, $\Dr(n)$ is strictly larger than $\TS_n^{\circ}$ by [\textit{loc. cit.} Theorem~4.5].

\subsection{Finite limits of strata}
We consider the following partial order of $\Delta$-matroids: $M'\leq M$ means that $Q_{M'}$ is a face of $Q_M$, and $M'\lessdot M$ means that $Q_{M'}$ is a facet of $Q_M$. We may view, for $w\in \Dr_M$, the polyhedral complex $\Delta_{M,w}$ as a finite poset, and hence we may form the inverse limit 
\begin{equation}
\label{eq:limitStrata}
\bS_{M,w} := \varprojlim_{Q\in \Delta_{M,w}} \bS_{M_Q}.
\end{equation}
If $\cB(M) = E(n)$, then we write $\bS_w$ for this limit. Finite inverse limits exist in the category of affine $\kk$-schemes because this category has a terminal object and pullbacks \cite[Proposition~5.21]{Awodey}. In fact, since $\Spec$ is left-adjoint to the global sections functor, it takes direct limits to inverse limits, so the affine coordinate ring of $\bS_{M,w}$ is the direct limit of the affine coordinate rings of the $\bS_{M}$'s.

\begin{proposition}
	\label{prop:mapInitToStratum}
	For any $w\in \TS_M$ and $u\in \RR^n$, the inclusion $B_{M_{u}^{w}}\subset B_{M}$ induces a morphism $\psi_{M,M_{u}^{w}}:\init_w\bS_M \to \bS_{M_{u}^{w}}$.  
\end{proposition}

\begin{proof}
We must show that the extension of $I_{M_{u}^{w}}$ to $B_{M}$ is contained in $\init_wI_M$. We use the quadric generators listed in Equation \eqref{eq:quadricGenerators}. We must show that $P_{M_{u}^{w}}(\mu,\nu) = 0$ or $P_{M_{u}^{w}}(\mu,\nu) = \init_wP_{M}(\mu,\nu)$. 

Suppose $P_{M_{u}^{w}}(\mu,\nu) \neq 0$.  Then either there is an $i_0\in \nu\setminus \mu$ such that $\mu\cup i_0$ and $\nu\setminus i_0$ are bases of $M_{u}^{w}$, or there is a $j_0\in \mu\setminus \nu$ such that $\mu\setminus j_0$ and $\nu\cup j_0$ are bases of $M_{u}^{w}$. The two situations are symmetric, so it suffices to consider the first one.  We must show that $q_{\mu\cup i}q_{\nu\setminus i}$ is a monomial of $\init_wP_{M}(\mu,\nu)$ if and only if $\mu\cup i$ and $ \nu\setminus i$ are bases of $M_{u}^{w}$ (resp. $q_{\mu\setminus j}q_{\nu\cup j}$ is a monomial of $\init_wP_{M}(\mu,\nu)$ if and only if $\mu\setminus j$ and $\nu\cup j$ are bases of $M_{u}^{w}$). 

Let $v_{\lambda} = w_{\lambda} + \bracket{u}{e_{\lambda}}$. Observe that, for any $i,j \in \nu\setminus \mu$, 
\begin{equation}
\label{eq:uw}
v_{\mu\cup j}  + v_{\nu\setminus j} -  v_{\mu\cup i}  - v_{\nu\setminus i} =  w_{\mu\cup j}  + w_{\nu\setminus j} - w_{\mu\cup i} - w_{\nu\setminus i}.
\end{equation}
Recall from the discussion after Equation \eqref{eq:quadricGens} that, since $|\mu\Delta\nu|\geq 4$, a monomial $q_{\eta}q_{\lambda}$ cannot appear in both sums in Equation \eqref{eq:quadricGenerators}, i.e., we need not be concerned with cancellations between the two sums. Let $q_{\mu\cup i}q_{\nu\setminus i}$ be a monomial in $P_{M}(\mu,\nu)$.  The term $q_{\mu\cup i}q_{\nu\setminus i}$ is a monomial of $\init_wP_{M}(\mu,\nu)$ if and only if 
\begin{align*}
&w_{\mu\cup i}  + w_{\nu\setminus i} \leq  w_{\mu\cup i'}  + w_{\nu\setminus i'} \hspace{10pt}\text{ and} \\
&w_{\mu\cup i}  + w_{\nu\setminus i} \leq  w_{\mu\setminus j'}  + w_{\nu\cup j'}
\end{align*}
for all $i',j'$, if and only if 
\begin{align*}
&v_{\mu\cup i}  + v_{\nu\setminus i} \leq  v_{\mu\cup i'}  + v_{\nu\setminus i'} \hspace{10pt}\text{ and} \\
&v_{\mu\cup i}  + v_{\nu\setminus i} \leq  v_{\mu\setminus j'}  + v_{\nu\cup j'}
\end{align*}
for all $i',j'$, if and only if $v_{\mu\cup i}  + v_{\nu\setminus i} = v_{\mu\cup i_0}  + v_{\nu\setminus i_0}$, if and only if $\mu\cup i$ and $\nu \setminus i$ are bases of $M_{u}^{w}$. A similar argument shows that, if $q_{\mu\setminus j}q_{\nu\cup j}$ is a monomial of $P_{M}(\mu,\nu)$, then   $q_{\mu\setminus j}q_{\nu\cup j}$ is a monomial of $\init_wP_{M}(\mu,\nu)$ if and only if $\mu\setminus j$ and $\nu\cup j$ are bases of $M_{u}^{w}$.
\end{proof}

\begin{theorem}
	\label{thm:closedImmersion}
	For any $w\in \TS_M$ the morphisms $\psi_{M,M_{u}^{w}}$ induce a closed immersion
	\begin{equation*}
	\psi_{M,w}:\init_w\bS_M \hookrightarrow \bS_{M,w}.
	\end{equation*}  
\end{theorem}

\begin{proof}
Clearly $\varphi_{M_{u}^{w},(M_{u}^{w})_{v}} \circ \psi_{M,M_{u}^{w}} = \psi_{M,(M_{u}^{w})_v}$, so $\psi_{M,w}$ is defined by the universal property of inverse limits. For each $q_{\lambda}$ for $\lambda \in \pB(M)$, there is a $Q\in \Delta_{M,w}$ such that $\lambda \in \pB(M_Q)$, in which case
\begin{equation*}
\psi_{M,w}^{\# }(q_{\lambda}) = \psi_{M,M_Q}^{\# }(q_{\lambda}) = q_{\lambda}.
\end{equation*} 
So the morphism 
$\psi_{M,w}^{\#}: \varinjlim_{Q\in \Delta_{M,w}} R_{M_Q} \to S_{M}^{-1} B_{M}/\init_wI_M$ is surjective, and therefore $\psi_{M,w}$ is a closed immersion. 
\end{proof}

\section{Symmetric matroids and initial degenerations of totally isotropic subspaces}
\label{sec:symmetricMatroids}

In this section, we provide the geometric description of Theorem \ref{thm:closedImmersion} as described in the introduction. This requires a notion of circuits for $\Delta$-matroids, which are better understood in the language of symmetric matroids. The reader interested in the applications of Theorem \ref{thm:closedImmersion} may skip to the next section. 

Let $J=\bn\sqcup \bn^*$. Define an involution $J\to J$ by $i\mapsto i^*$  and $(i^*)^*=i$ for $i\in \bn$. Given $\lambda \subset J$, let $\lambda^*=\set{i^*}{i\in \lambda}$. The subset $\lambda$ is \textit{admissible} if $\lambda \cap \lambda^* = \emptyset$, and is a \textit{transversal} if, additionally, $|\lambda|=n$. A \textit{symmetric matroid} $M$ is determined by a nonempty set of transversals $\cB(M)$ satisfying the \textit{symmetric exchange axiom}: for every $\mu,\nu\in \cB(M)$ and $i\in \mu \Delta \nu$, there is a $j\in \mu\Delta\nu$ such that $\mu\Delta ii^*jj^* \in \cB(M)$ (compare this to the symmetric exchange axiom used to define $\Delta$-matroids). An element of  $\cB(M)$ is called  a \textit{basis} of $M$. 

Symmetric matroids go by many different names in the literature, see \cite[\S 4]{BorovikGelfandWhite} and the references therein.  The data of a symmetric matroid on $J$ is also equivalent to the data of a $\Delta$-matroid on $\bn$. Given a $\Delta$-matroid $M$, its associated symmetric matroid $\overline{M}$ is defined by
\begin{equation*}
\cB(\overline{M}) = \set{\overline{\mu}}{\mu \in \cB(M)}.
\end{equation*}
Conversely, given a symmetric matroid $M$, its associated $\Delta$-matroid $\underline{M}$ is defined by 
\begin{equation*}
\cB(\underline{M}) = \set{\lambda \cap \bn}{\lambda\in \cB(M)}.
\end{equation*}

Let $M$ be a symmetric matroid on $J$. A subset of $J$  is \textit{independent} if it is contained in a basis, and \textit{dependent} otherwise.  A \textit{circuit} is an admissible minimal dependent subset, and the set of all circuits of $M$ is denoted by $\cC(M)$. We use a similar notation for ordinary matroids, i.e., if $M$ is an ordinary matroid, then we denote by $\cB(M)$ the bases of $M$ and the circuits of $M$ by $\cC(M)$.

Let $F$ be a totally isotropic subspace of $V.$ 
The symmetric matroid of $F$ is $\overline{M(F)}$ where $M(F)$ is the $\Delta$-matroid of $F$ defined in \eqref{eq:matroidF}. Explicitly, 
\begin{equation*}
\cB(\overline{M(F)}) = \set{\lambda \subset J}{\lambda \text{ is a transversal and } p_{\lambda}(F) \neq 0}.
\end{equation*}
The (ordinary) matroid of $F$, denoted by $M_{\A}(F)$, is given by
\begin{equation*}
\cB(M_{\A}(F)) = \set{\lambda\subset J}{p_{\lambda}(F) \neq 0}. 
\end{equation*}
That is, the bases of $\overline{M(F)}$ are just the bases of $M_{\A}(F)$ that are admissible.

Given an $m\times n$ matrix $X$ and subsets $\mu \subset [m]$, $\nu\subset\bn$, denote by $X[\mu,\nu]$ the submatrix of $X$ whose rows are indexed by $\mu$ and columns are indexed by $\nu$. We abbreviate $X[\mu,\mu]$ by $X[\mu]$. 

\begin{lemma}
	\label{lem:Cayley}
	Let $X$ be a skew-symmetric $n\times n$ matrix and let $\lambda$ be a subset of $[n]$ with $k$ elements.  
	Then, for distinct $i,j\in [n]$, we have
\begin{equation*}
\det(X[\lambda\setminus i, \lambda\setminus j]) = \left\{
\begin{array}{ll}
\pm \Pf(X[\lambda \setminus i])\Pf(X[\lambda\setminus j]) & \text{ if } k \text{ is odd}, \\
\pm \Pf(X[\lambda])\Pf(X[\lambda\setminus ij]) & \text{ if } k \text{ is even}.
\end{array}
\right.
\end{equation*}
Additionally, $\det(X[\nu,\mu]) = \pm\det(X[\mu,\nu])$.
\end{lemma}

\begin{proof}
	This was originally proved by Cayley, see \cite[p. 11]{Knuth} and the references therein. 
\end{proof}

\begin{lemma}
	\label{lem:detPfIsotropic}
	Let $X$ be a skew-symmetric $n\times n$ matrix and let $F$ be the totally isotropic subspace given by the row span of $W := [I|X]$. Let $\tau\in \cB(\overline{M(F)})$ such that  $|[n]\setminus \tau|$ is even. Then, for  $i\in \tau$ and $j\notin \tau$, we have
\begin{equation*}
\det(W[\bn,\tau \Delta ij]) = \pm \Pf(X[\bn \setminus \tau]) \Pf(X[\bn\setminus (\tau\Delta ii^*jj^*)]).
\end{equation*}
\end{lemma}

\begin{remark}
	In terms of Pl\"ucker and Wick coordinates, this proposition says that 
	\begin{equation*}
	p_{\tau\Delta ij}(F) = \pm q_{\tau\cap \bn}(F) q_{(\tau\Delta ii^*jj^*)\cap [n]}(F).
	\end{equation*}
\end{remark}

\begin{proof}
	First, observe that if $i=j^*$, then $\det(W[\bn,\tau \Delta ij]) $ and $ \Pf(X[\bn\setminus (\tau\Delta ii^*jj^*)])$ equal 0. Therefore, we assume that $i\neq j^*$. For brevity, let $\nu = \bn \setminus \tau$.
	Before considering the four cases, depending on whether $i$ or $j$ lie in $[n]$ or $[n]^*$, we record some useful formulas:
	\begin{enumerate}
		\item $\det(W[\bn, \sigma]) = \pm \det(X[\bn \setminus \sigma, \bn \cap \sigma^* ])$ for any $\sigma \subset J$ of size $n$. 
		\item $\bn \setminus (\tau\Delta ij) = \nu \cup ([n]\cap ii^*) \setminus jj^*$
		\item $\bn \cap (\tau\Delta ij)^* = \nu \cup ([n]\cap jj^*) \setminus ii^*$
	\end{enumerate} 
   If $i,j\in \bn$, then $i \notin \nu$, $j \in \nu$.  The above formulas and Lemma \ref{lem:Cayley} applied to $\lambda=\nu \cup i$ yield
   \begin{equation*}
   \det(W[\bn, \tau\Delta ij]) = \pm \det(X[\nu \Delta ij, \nu]) = \pm \Pf(X[\nu]) \Pf(X[\nu\Delta ij]).
   \end{equation*}
    If $i,j\in \bn^*$, then  $i^{*} \in \nu$, $j^{*} \notin \nu$ and we have 
   \begin{equation*}
   \det(W[\bn, \tau\Delta ij]) = \pm \det(X[\nu, \nu \Delta i^*j^*) ])  =\pm  \Pf(X[\nu]) \Pf(X[\nu\Delta i^*j^*)]).
   \end{equation*}
    If $i \in [n]$ and $j \in [n]^*$, then $i \notin \nu$, $j^* \notin \nu$ and we have 
   \begin{equation*}
   \det(W[\bn, \tau\Delta ij]) = \pm \det(X[\nu \cup i, \nu \cup j^*]) = \pm \Pf(X[\nu]) \Pf(X[\nu \cup ij^*]) 
   \end{equation*}
   If $i \in [n]^*$ and $j \in [n]$, then $i^* \in \nu$, $j \in \nu$ and we have 
   \begin{equation*}
   \det(W[\bn, \tau\Delta ij]) = \pm \det(X[\nu \setminus j, \nu \setminus i^*]) = \pm \Pf(X[\nu]) \Pf(X[\nu \setminus i^*j])  
   \end{equation*}
   In each of these cases, we get the desired result.
%
%
%
\end{proof}

If $\tau \in \cB(M_{\A}(F))$ and $j\in J\setminus \tau$, then $\tau\cup j$ contains a unique circuit $\gamma(\tau,j)$ of $M_{\A}(F)$ given by
\begin{equation}
\label{eq:fundamentalCircuit1}
\gamma(\tau,j) = \set{i\in\tau}{\tau\Delta ij \in \cB(M_{\A}(F)) } \cup\{j\} 
\end{equation}
Similarly, if $\tau \in \cB(\overline{M(F)})$ and $j\in J\setminus \tau$, then  $\tau \cup j$ contains a unique circuit $\overline{\gamma}(\tau,j)$ of $\overline{M(F)}$ given by
\begin{equation}
\label{eq:fundamentalCircuit2}
\overline{\gamma}(\tau,j) = \set{i\in\tau}{\tau \Delta ii^*jj^* \in \cB(\overline{M(F)}) } \cup \{j\}.
\end{equation}
Every circuit of $M_{\A}(F)$, resp. $\overline{M(F)}$, is of the form $\gamma(\tau,j)$, resp. $\overline{\gamma}(\tau,j)$. 

\begin{lemma}
	\label{lem:equalFundamentalCurcuits}
	If $\tau\in \cB(\overline{M(F)})$ and $j\in J \setminus  \tau$, then
	\begin{equation*}
	\gamma(\tau,j) =  \overline{\gamma}(\tau,j)
	\end{equation*}
\end{lemma}

\begin{proof}
Without loss of generality, assume that $\bn\in \cB(\overline{M(F)})$. Note that $n-|\tau\cap \bn|$ is even and $q_{\tau \cap \bn}(F) \neq 0$. By Lemma \ref{lem:detPfIsotropic}, we have that $i\in \gamma(\tau,j)\setminus j$ if and only if $p_{\tau \Delta ij}(F) \neq 0$, if and only if $q_{(\tau\Delta ii^*jj^*) \cap \bn}(F) \neq 0$, if and only if $\tau\Delta ii^*jj^* \in \cB(\overline{M(F)})$, if and only if $i\in \overline{\gamma}(\tau,j)\setminus j$.
\end{proof}
\noindent The set $\gamma(\tau,j)=\overline{\gamma}(\tau,j)$, is called the \textit{fundamental circuit} of the pair $(\tau,j)$.

Fix a Pl\"ucker vector $p(F)$ and a Wick vector $q(F)$ for $F$; assume that these are compatible in the sense that $p_{\lambda}(F) = q_{\lambda\cap [n]}(F) = 1$ for some $\lambda \in \cB(\overline{M(F)})$. View $F$ as a subscheme of $\Spec(\kk[y_{i},y_{i^*}])$ where $\kk[y_{i},y_{i^*}] = \kk[y_{i},y_{i^*} \, : \, i=1, \ldots, n]$, and let $I_F$ be the ideal of $F$. For $\tau \in \pB(M_{\A}(F))$ and $j\in J \setminus \tau$, there is, up to scaling, a unique linear form in $I_F$ with support $\gamma = \gamma(\tau,j)$. It is given by
\begin{equation*}
\ell_{\gamma}(F) =  (-1)^{\sg(j,j;\tau)}p_{\tau}(F)y_j + \sum_{i\in \tau} (-1)^{\sg(i,j;\tau)} p_{\tau\Delta ij}(F) \cdot y_i.
\end{equation*}
where $\sg$ is some sign function that is not important for us.
Now suppose $\tau\in \cB(\overline{M(F)})$ and $j\in J \setminus \tau$.  Then $\ell_{\gamma}(F)$ is a scalar multiple of the form 
\begin{equation*}
m_{\gamma}(F) =  (-1)^{\sg'(j,j;\tau)}q_{\tau\cap [n]}(F)y_j + \sum_{i \in \tau} (-1)^{\sg'(i,j,\tau)} q_{(\tau \Delta ii^*jj^*)\cap \bn}(F) \cdot y_i.
\end{equation*}
where $\sg'$ is some other sign function that is not important for us. Indeed, $\ell_{\gamma} = q_{\tau \cap \bn} m_{\gamma}$ by Lemma \ref{lem:detPfIsotropic}. Also note that $\supp(\ell_{\gamma}(F)) = \supp(m_{\gamma}(F)) = \gamma$ by \eqref{eq:fundamentalCircuit1}, \eqref{eq:fundamentalCircuit2} and Lemma \ref{lem:equalFundamentalCurcuits}, where $\supp$ denotes the support of a linear form.

\begin{proposition}
	\label{prop:gensIF}
The ideal of $F$ is generated by these linear forms, i.e.,
\begin{align*}
I_F = \langle \ell_{\gamma}(F) \, : \, \gamma\in \cC(M_{\A}(F))     \rangle =  \langle m_{\gamma}(F) \, : \, \gamma\in \cC(\overline{M(F)})     \rangle.
\end{align*}	
Moreover, $\set{\ell_{\gamma}(F)}{\gamma\in \cC(M_{\A}(F)}$ is a universal Gr\"obner basis for $I_F$. 
\end{proposition}

\begin{proof}
	The first equality and the last statement follows from \cite[Lemma~4.1.4]{MaclaganSturmfels2015}. Now consider the last equality. Each $m_{\gamma}(F)$ lies in $I_F$ because $m_{\gamma}(F)$ is a multiple of $\ell_{\gamma}(F)$. Fix $\tau\in \cB(\overline{M(F)})$. The set $\{m_{\gamma(\tau,j)}(F) \, : \,  j\notin \tau\}$ is a basis for the null space of $F$, and hence we get the second equality. 
\end{proof}

We end this section by making precise the geometric characterization of Theorem \ref{thm:closedImmersion} presented in the introduction.   Let $M$ be an even $\Delta$-matroid, $x$ be a $\kk$-point of $\init_w\bS_{M}$, and $\KK = \kk(\!(t^{\RR})\!)$. By surjectivity of exploded tropicalization \cite{Payne}, there is a $\KK$-point $q$ of $\bS_M$ such that $\ETrop(q) = x$. Let $F_{q} \subset \Spec(\KK[y_i,y_{i^*}])$ be the totally isotropic subspace with Wick vector $q$, and $F_q^{\circ}$ the intersection of $F_{q}$ with the dense torus $\Spec(\KK[y_i^{\pm},y_{i^*}^{\pm}])$.  By Proposition \ref{prop:initIsotropic} below, if $(-u,u) \in \Trop(F_{q}^{\circ})$ then  $\overline{\init_{(-u,u)}F_q^{\circ}}$ is totally isotropic, and a realization of $M_{u}^{w}$. Therefore, the map $\init_{w}\bS_{M} \to \bS_{M_{u}^{w}}$ takes $x$ to (the Wick vector of) $\overline{\init_{(-u,u)}F_q^{\circ}}$.

\begin{proposition}
	\label{prop:initIsotropic} 
	Suppose $(-u,u) \in \Trop(F_q^{\circ})$.   The closure of $\init_{(-u,u)}F_q^{\circ}$ in $\kk^{2n}$ is totally isotropic, and realizes the $\Delta$-matroid $M_{u}^{w}$. 
\end{proposition}
\begin{proof}
	By Proposition \ref{prop:mapInitToStratum}, the vector $\overline{x} = (x_{\mu} \, : \, \mu\in \cB(M_{u}^{w})) \times 0 \in \kk^{\cB(M_{u}^{w})}\times \kk^{E(n)\setminus \cB(M_{u}^{w})}$ lies in $\bS_{M_u^w}$. That is, $\overline{x}$ is the Wick vector of a totally isotropic subspace $F_{\overline{x}}$ realizing $\bS_{M_u^w}$. We claim that $\overline{\init_{(-u,u)}F_{q}^{\circ}} = F_{\overline{x}}$. Let $I_{F_q}$ be the ideal of $F_{q}$ in $\KK[y_{i},y_{i^*}]$ and $I_{F_{\overline{x}}}$ the ideal of $F_{\overline{x}}$ in $\kk[y_{i},y_{i^*}]$. The linear forms $\ell_{\gamma}(F_q)$ form a universal Gr\"obner basis for $I_{F_q}$, so
	\begin{equation*}
	\init_{(-u,u)} I_{F_q} = \langle \init_{(-u,u)} \ell_{\gamma}(F_q) \, :\, \gamma \in \cC(M_{\A}(F_q))\rangle \subset \kk[y_{i},y_{i^*}].
	\end{equation*}
	Because $I_{F_{\overline{x}}} $ and $ \init_{(-u,u)} I_{F_q}$ define $n$-dimensional linear subschemes of $\Spec(\kk[y_{i},y_{i^*}])$, it suffices to show $I_{F_{\overline{x}}} \subset \init_{(-u,u)} I_{F_q}$. Given $\tau\in \cB(\overline{M_{u}^{w}})$ and $j\in J\setminus  \tau$, let $\gamma$ be the fundamental circuit of $(\tau,j)$ in $\overline{M}$, and $\gamma'$ the fundamental circuit of $(\tau,j)$ in $\overline{M_{u}^{w}}$. That is,
	\begin{align*}
	\gamma &= \set{i\in \tau}{\tau \Delta ii^*jj^* \in \cB(\overline{M})}\cup\{j\}, \\
	\gamma' &= \set{i\in \tau}{\tau \Delta ii^*jj^* \in \cB(\overline{M}) \text{ and } \bracket{u}{e_{(\tau \Delta ii^*jj^*)\cap [n]}} + w_{(\tau \Delta ii^*jj^*)\cap [n]}  \text{ is minimal}} \cup\{j\}.
	\end{align*}
	We have
	\begin{equation*}
	\init_{(-u,u)} m_{\gamma}(F_{q}) =  (-1)^{\sg'(j,j;\tau)}x_{\tau\cap [n]}y_j(F_{q}) +  \sum_{i\in \gamma''}  (-1)^{\sg'(i,j,\tau)} x_{(\tau \Delta ii^*jj^*)\cap \bn}(F_{q}) \cdot y_i,
	\end{equation*}
	where
	\begin{align*}
	\gamma'' &= \set{i\in \tau \setminus j}{\tau \Delta ii^*jj^* \in \cB(\overline{M}) \text{ and } \bracket{(-u,u)}{e_i} + w_{(\tau \Delta ii^*jj^*)\cap [n]}  \text{ is minimal}}.
	\end{align*}
	By the equality 
	\begin{equation*}
	\bracket{u}{e_{(\tau \Delta ii^*jj^*)\cap [n]}} - \bracket{(-u,u)}{e_i}  = \bracket{u}{e_{(\tau \Delta jj^*) \cap [n]}} 
	\end{equation*}
	 we see that $\supp(	\init_{(-u,u)} m_{\gamma}(F_{q})) = \gamma'' \cup j = \gamma'$, and therefore  $\init_{(-u,u)} m_{\gamma}(F_{q}) = m_{\gamma'}(F_{\overline{x}})$. By Proposition \ref{prop:gensIF} and because every circuit $\gamma'$ of $\overline{M_{u}^{w}}$ is a fundamental circuit of some pair $(\tau,j)$, we see that $I_{F_{\overline{x}}}\subset  \init_{(-u,u)} I_{F_q}$, as required.
\end{proof}

\section{Affine coordinates for strata} 
\label{sec:TechniquesLimitsStrata}

\subsection{Affine coordinates}
Throughout this subsection, we fix an even $\Delta$-matroid $M$ on $\bn$ that has $\bn$ as a basis. This means that $\bS_{M} $ is contained in open cell $\cU$ from \eqref{eq:affineChart}, so every $F\in \bS_{M}(\kk)$ is the row span of a $\kk$-valued $n\times 2n$ matrix of the form $[I_n|X_F]$ where $X_F$ is skew-symmetric. Let $X$ be the skew-symmetric matrix of variables 
\begin{equation*}
X = \begin{bmatrix}
0 & x_{01} & x_{02} & \cdots & x_{0,n-1} \\
-x_{01} & 0 & x_{12} & \cdots & x_{1,n-1} \\
-x_{02} & -x_{12} & 0 & \cdots & x_{2,n-1} \\
\cdots & \cdots & \cdots & \cdots &\cdots \\
-x_{0,n-1} & -x_{1,n-1} & -x_{2,n-1} & \cdots & 0 
\end{bmatrix}.
\end{equation*}
As before, let $X[\lambda]$ denote the submatrix of $X$ whose rows and columns are indexed by $\lambda$, and $\Pf(X[\lambda])$ the Pfaffian of $X[\lambda]$.  Define
\begin{itemize}[noitemsep]
	\item[-] $B_{M}^{x} = \kk[x_{ij} \, : \, \bn \setminus ij \in  \cB(M)]$;
	\item[-] $I_M^x = \langle \Pf(X[\bn\setminus \lambda]) \, : \, \lambda \in E(n) \setminus \cB(M) \rangle \cap B_M^x$;
	\item[-] $S_{M}^x$ the multiplicative semigroup generated by $\set{\pi(\Pf(X[\bn\setminus \lambda]))}{\lambda \in \cB(M)}$  where  $\pi:\kk[x_{ij}] \to \kk[x_{ij}]/\langle x_{ij} \, : \, \bn\setminus \{ij\}\in E(n) \setminus \cB(M) \rangle  \cong  B_M^x$ is the quotient map. 
\end{itemize}
The coordinate ring of $\bS_M$ is isomorphic to 
\begin{equation*}
R_{M}^{x} = (S_{M}^x)^{-1}B_M^x/I_M^x.
\end{equation*}

\begin{proposition}
	\label{prop:mapXcoordinates}
	Suppose $\bn$ is a basis of both $M$ and $M_{u}$. The induced morphism $\varphi_{M,M_u}^{\#}: R_{M_u}^x \to R_{M}^x$ is given by $\varphi_{M,M_u}^{\#}(x_{ij}) = x_{ij}$. 
\end{proposition}

\begin{proof}
	Set
	\begin{equation*}		
	\tilde{R}_{M} = S_M^{-1}\kk[q_{\lambda}/q_{\bn} \, | \, \lambda \in \cB(M)]/I_M 
	\end{equation*}
	The map  $q_{\lambda}/q_{\bn} \mapsto \Pf_{\lambda}(X)$ determines an isomorphism $\theta_M:\tilde{R}_{M} \to R_{M}^x$. 	By Proposition \ref{prop:mapPlucker}, the morphism $\varphi_{M,M_u}:\bS_{M} \to \bS_{M_u}$ induces the ring homomorphism $\psi: \tilde{R}_{M_u} \to \tilde{R}_{M}$ given by $\psi(q_{\lambda}/q_{\bn}) = q_{\lambda}/q_{\bn}$. The induced map $\varphi_{M,M_u}^{\#}$ is equal to $\theta_{M_u}\circ \psi \circ \theta_{M}^{-1}$, which takes $x_{ij}$ to $x_{ij}$, as required.
\end{proof}

We illustrate the affine coordinates construction with the following example, which we use in the proof of Theorem \ref{thm:schonSpinor}.
\begin{example}
	\label{ex:SpecialDeltaMatroid}
	Suppose $K$ is the even $\Delta$-matroid on $[5]$ with 
	\begin{equation*}
	\cB(K) = \{0, 1, 012, 013, 014, 034, 134, 01234  \} 	
	\end{equation*}
	Then
	\begin{equation*}
	B_{K}^x=\kk[x_{02}, x_{12}, x_{23}, x_{24}, x_{34}], \hspace{10pt}  I_K^x = \langle 0 \rangle, \hspace{10pt} \text{ and } \hspace{10pt}  S_K^{x} = \langle x_{02}, x_{12}, x_{23}, x_{24}, x_{34}\rangle_{\semigp}.
	\end{equation*}
	and hence $R_K^x = \kk[x_{02}^{\pm}, x_{12}^{\pm}, x_{23}^{\pm}, x_{24}^{\pm}, x_{34}^{\pm}]$. This shows that $\bS_K \cong \GGm^5$. The polytope $Q_{K}$ has two facets not contained in $\partial \Delta(5)$ which are defined by the vectors $u = (1,1,-1,1,-1)$ and $v = (1,1,-1,-1,1)$. The initial $\Delta$-matroids $K_u$ and $K_v$ have bases
	\begin{equation*}
	\cB(K_u) = \{0, 1, 012, 014, 034, 134, 01234  \} \hspace{10pt} \cB(K_v) = \{0, 1, 012, 013, 034, 134, 01234  \}.
	\end{equation*} 
	The coordinate rings of $\bS_{K_u}$ and $\bS_{K_v}$ are
	\begin{equation*}
	R_{K_u}^x = \kk[x_{02}^{\pm}, x_{12}^{\pm}, x_{23}^{\pm}, x_{34}^{\pm}], \hspace{20pt} 
	R_{K_v}^x = \kk[x_{02}^{\pm}, x_{12}^{\pm}, x_{24}^{\pm}, x_{34}^{\pm}].
	\end{equation*}
	Therefore, $\bS_{K_u}$ and $\bS_{K_v}$ are isomorphic to $\GGm^4$, The morphisms 
	\begin{equation*}
	\varphi_{K,K_u}:\bS_{K}\to \bS_{K_u} \hspace{20pt} \varphi_{K,K_v}:\bS_{K}\to \bS_{K_v}
	\end{equation*}
	may be identified with coordinate projections of tori; in particular, they are smooth and surjective with connected fibers.
\end{example}

\begin{lemma}
	\label{lem:closedImmersionBoundary}
	The morphism
	$(\varphi_{K,K_u}, \varphi_{K_u,K_v}):\bS_{K} \to \bS_{K_u}\times \bS_{K_v}$ is a closed immersion.
\end{lemma}

\begin{proof}
The induced map on coordinate rings is 
\begin{equation*}
\varphi_{K,K_u}^{\#} \otimes \varphi_{K,K_v}^{\#} : R_{K_u}^x \otimes_{\kk} R_{K_v}^x \to R_{K}^x
\end{equation*}
which is surjective; see the explicit description of these maps above. 
\end{proof}

\subsection{Inverse limits over a graph}
\label{sec:InverseLimitsGraphs}

In this subsection we  show that, to compute the inverse limit $\bS_{M,w}$ in \eqref{eq:limitStrata}, we do not need the full poset $\Delta_{M,w}$, just those cells of codimension 0 and 1. In other words, we show that $\bS_{M,w}$ may be computed as a limit over a diagram recorded by  the adjacency graph $\Gamma_{M,w}$.  We begin by recalling this general construction, see \cite[Appendix~A]{Corey17} for details.

Let $\cC$ be a category that has finite limits; by \cite[Proposition~5.21]{Awodey}, it is necessary and sufficient that $\cC$ has fiber products and a terminal object.  Let $G$ be a connected graph, possibly with loops or multiple edges. We view each edge as a pair of half-edges. Define a quiver $Q(G)$ in the following way. The set of vertices of $Q(G)$ is $V(G)\cup E(G)$; write $q_v$ (resp. $q_e$) for the vertex of $Q(V)$ corresponding to the vertex $v$ (resp. edge $e$). For each half-edge $h\in e$ adjacent to $v$, there is an arrow $q_v\to q_e$. Viewing $Q(G)$ as a category in the usual way, a \textit{diagram}  of type $Q(G)$ in $\cC$ is a functor $X:Q(G) \to \cC$.

Let $\Gamma_{M,w}$ be the adjacency graph of a matroid subdivision $\Delta_{M,w}$. Let $M_v$, resp. $M_e$, denote the $\Delta$-matroid corresponding to the vertex $v$, resp. edge $e$, of $\Gamma_{M,w}$, and $\varphi_{M_v,M_e}:\bS_{M_v} \to \bS_{M_e}$ whenever $e$ is incident to $v$. The data of $\bS_{M_v}, \bS_{M_e}$, and $\varphi_{M_v,M_e}$ defines a diagram of type  $Q(\Gamma_{M,w})$ in $\kSch$.

Let  $(\Delta_{M,w})^{\Top}$ be the  set of top-dimensional cells of $\Delta_{M,w}$, and $A$ be a nonempty subset of  $(\Delta_{M,w})^{\Top}$. We isolate some properties of $A$ that allows for different ways to study limits of spinor strata over full subgraphs $\Gamma_{M,w}[A]:= \Gamma_{M,w}[\set{v_Q}{Q\in A}]$ of the adjacency graph $\Gamma_{M,w}$ and their relation to initial degenerations of spinor varieties. 
\begin{itemize}[noitemsep]
\item[-] The subset $A$ is  \textit{basis-covering} if   
\begin{equation*}
\bigcup_{Q\in A} \cB(M_Q) = \cB(M).
\end{equation*}
\item[-] The subset $A$ is \textit{basis-intersecting} 
\begin{equation*}
\bigcap_{Q\in A} \cB(M_Q) \neq \emptyset.
\end{equation*}
\item[-]The subset $A$ is \textit{basis-connecting} if, for each $\beta \in \bigcup_{Q\in A} \cB(M_Q)$, the induced subgraph
\begin{equation*}
\Gamma_{M,w}[\set{v_Q}{Q\in A \text{ and } \beta\in \cB(M_Q)}]
\end{equation*}
is connected.
\end{itemize}

\begin{proposition}
	\label{prop:subdGood}
	For any even $\Delta$-matroid $M$ and $w\in \Dr_{M}$, the set $(\Delta_{M,w})^{\Top}$ is basis-covering and basis-connecting. Moreover,
	\begin{equation*}
	\bS_{M,w} \cong \varprojlim_{\Gamma_{M,w}} \bS_{M_Q}.
	\end{equation*}
\end{proposition}

\begin{proof}
	The analog of this proposition for limits of thin Schubert cells in the Grassmannian follows from \cite[Propositions~C.11-12]{Corey17}. The proof in the spinor strata case is analogous. 
\end{proof}

Next, we show how to compute the coordinate ring of inverse limits over the graphs $\Gamma_{M,w}$. First, in Wick coordinates, define
\begin{itemize}[noitemsep]
	\item[-] $B=\kk[q_{\lambda} \, : \, \lambda \in \bigcup_{Q\in A} \cB(M_Q)]$;
	\item[-] $I = \sum_{Q\in A} I_{M_Q} \cdot B$;
	\item[-] $S$ is the multiplicative semigroup generated by $\set{q_{\lambda}}{Q\in A \text{ and } \lambda \in \cB(M_Q)}$.
\end{itemize} 
Set $R(A)= S^{-1} B/I$. 

For affine coordinates, we must assume $A$ is basis-intersecting, say $\lambda$ lies in the intersection of all bases sets. After twisting each $M_Q$ by $\bn\setminus \lambda$, we may assume that $\lambda=\bn$.  Define
\begin{itemize}[noitemsep]
	\item[-] $B^x=\kk[x_{ij} \, : \, \bn\setminus ij \in \bigcup_{Q\in A} \cB(M_Q)]$;
	\item[-] $I^x = \sum_{Q\in A} I_{M_Q}^x \cdot B^x$;
	\item[-] $S^x$ is the multiplicative semigroup generated by
	\begin{equation*}
	\set{\pi_{Q}(\Pf(X[\lambda]))}{Q\in A \text{ and } \bn\setminus \lambda \in \cB(M_Q)}
	\end{equation*}
	where $\pi_Q$ is the composition 
	\begin{equation*}
	\kk[x_{ij}] \to \kk[x_{ij}]/\langle x_{ij} \, : \, \bn\setminus ij \notin \cB(M_Q) \rangle  \cong  B_{M_Q}^x \subset B^x.
	\end{equation*}
\end{itemize} 
Set $R^x(A)= (S^x)^{-1} B^x/I^x$.

\begin{proposition}
	\label{prop:coordinateRingLimit}
	If $A$ is basis-connecting,
	 then 
	\begin{equation*}
	\varinjlim_{\Gamma_{M,w}[A]} R_{M_Q} \cong R(A).
	\end{equation*}
	If, in addition, $\bn \in \cB(M_Q)$ for all $Q\in A$ (so $A$ is basis-intersecting), then
	\begin{equation*}
	\varinjlim_{\Gamma_{M,w}[A]} R_{M_Q}^x \cong R^x(A).
	\end{equation*}
\end{proposition}

\begin{proof}
	We prove the second statement, the first one is similar, see also \cite[Proposition~3.7]{Corey17}.  
	Let $\hat{A}$ denote the set of all $Q\in \Delta_{M,w}$ such that either $v_Q\in V(\Gamma_{M,w}[A])$ or $e_Q\in E(\Gamma_{M,w}[A])$. For each $Q\in \hat{A}$, we have a ring homomorphism $R_{M_Q}^{x} \to R^x(A)$ defined by $x_{ij}\mapsto x_{ij}$. These piece together to give a ring homomorphism $\Psi:\varinjlim_{\Gamma_{M,w}[A]} R_{M_Q}^{x} \to R^{x}(A)$. 
	We define a ring homomorphism $\Theta$ that is an inverse to $\Psi$.
	
	If $\bn\setminus ij \in \cB(M_Q)$ for some $Q\in A$ then set $\Theta(x_{ij}) = \varphi_{M_{Q}}^{\#}(x_{ij})$. Suppose $Q'\in A$ is another cell such that $\bn\setminus ij \in \cB(M_{Q'})$.  If $Q'' := Q\cap Q'$  lies in $\hat{A}$,  then $\varphi_{M_{Q}}^{\#}(x_{ij}) = \varphi_{M_{Q''}}^{\#}(x_{ij}) = \varphi_{M_{Q'}}^{\#}(x_{ij})$. Because $A$ is basis-connecting, there is a sequence $Q=Q_0, Q_1,\ldots, Q_k = Q'$ such that, for each $\ell$, we have $[n]\setminus ij \in \cB(M_{Q_{\ell}})$ and  $Q_{\ell} \cap Q_{\ell+1}$ is a facet of $Q_{\ell}$ and $Q_{\ell+1}$. Then
	\begin{equation*}
	\varphi_{M_{Q}}^{\#}(x_{ij}) = \varphi_{M_{Q_0\cap Q_{1}}}^{\#}(x_{ij}) = \varphi_{M_{Q_1}}^{\#}(x_{ij}) = \cdots = \varphi_{M_{Q'}}^{\#}(x_{ij})
	\end{equation*}
	so $\Theta:B^x \to \varinjlim_{\Gamma_{M,w}[A]} R_{M_Q}^{x}$ is well defined. Also, $\Theta(z)$ is invertible for any $z\in S^x$. Finally, we need to show that $I^{x} \subset \ker(\Theta)$. It suffices to show that $\Theta(zf) = 0$ for $z\in (S^x)^{-1}B^x$ and $f\in I_{M_Q}^x$ for some $Q\in A$, which follows from the fact that $\Theta(af) = \Theta(a)\varphi_{M_Q}^{\#}(f) = 0$. Therefore, $\Theta$ induces a ring homomorphism $\Theta: R^x(A) \to \varinjlim_{\Gamma_{M,w}[A]} R_{M_{Q}}^{x}$, which is clearly an inverse to $\Psi$. 
\end{proof}

Being basis-covering implies that the morphism $\init_w\bS_{M} \to \varprojlim_{\Gamma_{M,w}[A]}  \bS_{M_Q}$ is a closed immersion. 
\begin{proposition}
	\label{prop:basisCovering}
	If  $A\subset (\Delta_{M,w})^{\Top}$ is basis-covering, then  the morphisms $\init_w\bS_{M} \to \bS_{M_Q}$ induce a closed immersion  $\init_w\bS_{M} \hookrightarrow \varprojlim_{\Gamma_{M,w}[A]} \bS_{M_Q}$.
\end{proposition}

\begin{proof}
	The morphism $\init_w\bS_{M} \hookrightarrow \varprojlim_{\Gamma_{M,w}[A]} \bS_{M_Q}$ is defined by the universal property of inverse limits, and the induced morphism on coordinate rings is surjective because every basis of $M$ is a basis of some $\cB(M_Q)$ for $Q \in A$. 
\end{proof}

\begin{proposition}
	\label{prop:basisAll3}
	If $A\subset (\Delta_{w})^{\Top}$ is basis-covering, basis-intersecting, and basis-connecting, then $\varprojlim_{\Gamma_{w}[A]}\bS_{M_Q} $ is isomorphic to an open dense subvariety of $\AA^{{n\choose 2}}$. In particular, $\init_w \bS_{n}^{\circ} \cong \varprojlim_{\Gamma_{w}[A]} \bS_{M_Q}$ and these are smooth and irreducible of dimension ${n\choose 2}$. 
\end{proposition}

\begin{proof}
	Being basis-intersecting and basis-connecting implies that $\varprojlim_{\Gamma_{w}[A]}\bS_{M_Q} $ is isomorphic to a locally-closed subvariety of $\AA^{{n\choose 2}}$ by Proposition \ref{prop:coordinateRingLimit}; in particular, the dimension of $\varprojlim_{\Gamma_{w}[A]}\bS_{M_Q} $ is at most ${n\choose 2}$. Because $A$ is basis-covering, there is a closed immersion $\init_w\bS_{n}^{\circ} \hookrightarrow \varprojlim_{\Gamma_{w}[A]}\bS_{M_Q}$ by Proposition \ref{prop:basisCovering}, and the dimension of $\init_{w}\bS_{n}^{\circ}$ is ${n\choose 2}$ since this is a flat degeneration of $\bS_{n}^{\circ}$. Therefore, the dimension of $\varprojlim_{\Gamma_{w}[A]}\bS_{M_Q}$  is exactly ${n\choose 2}$, so $\varprojlim_{\Gamma_{w}[A]}\bS_{M_Q}$  is an open subvariety of $\AA^{{n\choose 2}}$. The last statement follows from the fact that a closed immersion between integral affine schemes of the same dimension is an isomorphism, see \cite[Proposition~A.8]{Corey17}.
\end{proof}

\begin{corollary}
	\label{cor:subdGood}
	If $(\Delta_{w})^{\Top}$ is basis-intersecting, then $\init_w \bS_n^{\circ} \cong\bS_{w}$. Furthermore, they are smooth and irreducible varieties of dimension ${n\choose 2}$. 
\end{corollary}

\begin{proof}
	By Proposition \ref{prop:subdGood}, $(\Delta_{w})^{\Top}$ is already basis-covering and basis-connecting. The corollary now follows from Proposition \ref{prop:basisAll3}. 
\end{proof}

\begin{example}
	\label{ex:subd4}
	Up to $W(D_4)$-symmetry, there are only 3 matroid subdivisions of $\Delta(4)$. One is the trivial subdivision, and the adjacency graphs for the other two are recorded in Table \ref{fig:subdTS4}.  From this data, and Corollary \ref{cor:subdGood}, we conclude that $\init_w\bS_4^{\circ}$ is smooth and irreducible for all $w\in \TS_4^{\circ}$. 
	
	\begin{table}[h!]
		\centering\renewcommand\cellalign{lc}
		\begin{tabular}{|c|c|m{5.7cm}|}
			\hline
			$w$ & Adjacency graph & Matroids \\
			\hline
			\hline
			$r_3$ & \raisebox{-\totalheight+4mm}{\includegraphics[width=2.3cm]{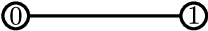}}
			&   \makecell{$M_{0}: \{\emptyset, 01, 02, 03, 13, 23, 0123\}$,  \\
				$M_{1}: \{\emptyset, 01, 02, 12, 13, 23, 0123\}$ }  \\
			\hline
			$r_2+r_3$ &  \raisebox{-\totalheight+10mm}{\includegraphics[width=2.5cm]{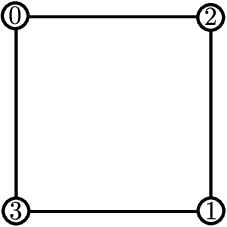}}
			& \makecell{$M_{0}:\{\emptyset, 01, 12, 13, 23, 0123\}$,\\  $M_{1}:\{\emptyset, 01, 02, 03, 23, 0123\}$,\\ $M_{2}:\{\emptyset, 01, 02, 12, 23, 0123\}$, \\ $M_{3}:\{\emptyset, 01, 03, 13, 23, 0123\}$}\\
			\hline
		\end{tabular}
		\caption{Matroidal subdivisions of $\Delta(4)$; the rays $r_i$ are listed in \eqref{eq:rays4}}
		\label{fig:subdTS4}
	\end{table}	
\end{example}

\begin{proposition}
\label{prop:subdAllBut1}
Suppose that there is a $Q \in (\Delta_{w})^{\Top}$ such that $v_Q$ is adjacent to exactly $2$ vertices $v_{Q_1}$, $v_{Q_2}$ which are themselves adjacent, and $A=(\Delta_{w})^{\Top} \setminus \{Q\}$ is basis-covering.  Then $A$ is basis-connecting. If $A$ is also basis-intersecting, then $\init_w \bS_n^{\circ}$ is smooth and irreducible.
\end{proposition}

\begin{proof}
First, we show that $A$ is basis-connecting. Let $\beta \in E(n)$ and
\begin{align*}
&H_{\beta} = \Gamma_{w}[\set{v_{Q'}}{Q'\in (\Delta_{w})^{\Top} \text{ and } \beta\in \cB(M_{Q'})}], \\
&H_{\beta}' = \Gamma_{w}[\set{v_{Q'}}{Q'\in A \text{ and } \beta\in \cB(M_{Q'})}].
\end{align*}
By Proposition \ref{prop:subdGood}, $H_{\beta}$ is connected. If $\beta \notin \cB(M_Q)$, then $H_{\beta}' = H_{\beta}$ and hence connected. Suppose $\beta\in \cB(M_{Q})$.  By hypothesis, there is an edge between $v_{Q_1}$ and $v_{Q_2}$.   If $v_{Q_1},v_{Q_2}\in V(H_{\beta}')$, then there is still an edge of $H_{\beta}'$ between $v_{Q_1},v_{Q_2}$, and hence $H_{\beta}'$ is connected. Otherwise, either $v_{Q_1}$ or $v_{Q_2}$ is in $V(H_{\beta}')$ since $H_{\beta}$ is connected. This means that $v_{Q}$ is a leaf-vertex of $H_{\beta}$, so $H_{\beta}'$ is connected. Therefore, $A$ is basis-connected. The last statement now follows from Proposition \ref{prop:basisAll3}.  
\end{proof}

\begin{lemma}
	\label{lem:pullbackClosedImmersion}
	Suppose we have a  pullback diagram of finite-dimensional affine schemes
	\begin{equation*}
	\xymatrix
	{{W\times_Z X}\ar[r]^-{} \ar[d]_-{f'} & X \ar[d]^f\\
		{W} \ar[r]^{} & Z}
	\end{equation*} 
	where $f:X\to Z$ is a closed immersion, $W$ is irreducible, and $W\times_Z X$ has the same dimension as $W$. Then $W\times_Z X \cong W$.  
\end{lemma}

\begin{proof}
	The morphism $f':W\times_Z X \to W$ is a closed immersion because closed immersions are preserved by arbitrary base change \cite[Exercise~II.3.11(a)]{Hartshorne}, so $f'$ is an isomorphism by \cite[Proposition~A.8]{Corey17}.
\end{proof}

\begin{theorem}
	\label{thm:schonSpinor}
For $n\leq 5$, the initial degenerations of $\bS_n^{\circ}$ are smooth and irreducible. 
\end{theorem}

\begin{proof}
	For $n=1,2,3$, $\bS_n = \PP(\kk^{E(n)})$, so the initial degenerations of $\bS_n^{\circ}$ are clearly smooth and irreducible. 	The $n=4$ case is handled in Example \ref{ex:subd4}, but we can also see this directly because $I_4$ is a principal ideal. Up to $W(D_4)$-symmetry, there are only $3$ cones of $\TS_4^{\circ}$. Representative weight vectors are
	\begin{equation*}
	w_0=0, \hspace{20pt} w_1 = f_{03}+f_{12}, \hspace{20pt} w_{2} = f_{02}+f_{12}+f_{03}+f_{13}.
	\end{equation*} 
	The initial ideals are
	\begin{align*}
	&\init_{w_0}I_4 = \langle q_{\emptyset} q_{0123} - q_{01}q_{23} + q_{02}q_{13} - q_{03}q_{12} \rangle, \\
	&\init_{w_1}I_4 = \langle q_{\emptyset} q_{0123} - q_{01}q_{23} + q_{02}q_{13} \rangle, \\
	&\init_{w_2}I_4 = \langle q_{\emptyset} q_{0123} - q_{01}q_{23} \rangle.
	\end{align*}
	These define smooth and irreducible varieties in the dense torus of $\PP(\kk^{E(4)}) \cong \PP^{7}$.

Now consider $n=5$. The subdivisions are listed in Appendix \ref{appendix:subdivisionsn5}. If $\Delta_w$ is the subdivision 1, 2, 3, 5, 8, 11, 13, or 16, then $\init_w\bS_5^{\circ}$ is smooth and irreducible by Proposition \ref{prop:subdGood}. Next, suppose $\Delta_{w}$ is one of the subdivisions 0, 4, 6, 7, 9, 12, 14, 17, or 18. Observe that, for each of these subdivisions, $\{0\}$ is a basis of $M_Q$ for all but one $Q \in  (\Delta_{w})^{\Top}$; the matroids missing $\{0\}$ are $M_{0}$, $M_{4}$, $M_{5}$, $M_{4}$, $M_{3}$, $M_{9}$, $M_{4}$, $M_{11}$, $M_{4}$, respectively. For each of these subdivisions, $v_{Q_{M_i}}$ has exactly two adjacent vertices, which are themselves adjacent. So $\init_w\bS_5^{\circ}$ is smooth and irreducible by Proposition \ref{prop:subdAllBut1}. The only subdivisions left are 10 and 15.
	
Consider subdivision 10 in Table \ref{tab:dim4}. Denote the $i$-th vertex by $v_i$, its polytope by $Q_i$, and its $\Delta$-matroid by $M_i$. For the edge between $v_i$ and $v_j$, denote by $Q_{ij}$ the corresponding polytope, and $M_{ij}$ its $\Delta$-matroid.  The set $A=\{Q_0,Q_1,Q_2,Q_3,Q_4\}$ is basis-connecting. Also, we see that $I_{M_i}^{x} = \langle 0 \rangle$ for $i=1,3,4$ and $I_{M_i}^{x} = \langle x_{02}x_{13}-x_{03}x_{12} \rangle$ for $i=0,2$. Therefore, the coordinate ring of $\varprojlim_{H}\bS_{M}$, where $H=\Gamma_{w}[v_{0},v_{1},v_{2},v_{3},v_{4}]$, is isomorphic to
	\begin{align*}
	R^x(A) = (S^x)^{-1}\kk[x_{ij}^{\pm} \, | \, 0\leq i<j\leq 4] / (x_{02}x_{13}-x_{03}x_{12})
	\end{align*}
	where $S^x$ is some finitely-generated multiplicative semigroup. Because we can solve for, e.g., $x_{02}$, we see that $\varprojlim_{H}\bS_{M}$ is isomorphic to an open subvariety of $\GGm^{9}$. Therefore, $\varprojlim_{H}\bS_{M}$ is smooth and irreducible of dimension 9.  Next, consider the pair $(Q_5, Q_{25})$. Twist $M_{5}$ and $M_{25}$ by $1234$ to get the $\Delta$-matroids $N$ and $N'$ with bases
	\begin{align*}
	\cB(N) &= \{1, 2, 3, 012, 013, 123, 124, 134, 234, 01234 \}, \\
	 \cB(N') &= \{1, 2, 3, 012, 013, 124, 134, 234, 01234 \}.
	\end{align*}
	The rings of $R_{N}^{x}$ and $R_{N'}^{x}$ are 
	\begin{align*}
	R_{N}^{x} = (S_N^x)^{-1} \kk[x_{01}^{\pm}, x_{02}^{\pm}, x_{03}^{\pm}, x_{04}^{\pm}, x_{24}^{\pm}, x_{34}^{\pm}], \hspace{10pt}  R_{N'}^{x} = (S_{N'}^x)^{-1}\kk[x_{01}^{\pm}, x_{02}^{\pm}, x_{03}^{\pm}, x_{24}^{\pm}, x_{34}^{\pm}]
	\end{align*}
	and the map $\varphi_{N,N'}^{\#}: R_{N'}^{x} \to R_{N}^{x}$ is given by $x_{ij}\mapsto x_{ij}$, so $\varphi_{N,N'}:\bS_{N} \to \bS_{N'}$ is smooth and dominant with connected fibers. Set $H' = \Gamma_w[v_0,\ldots, v_5]$.  By \cite[Proposition~A.5]{Corey17}, we have an isomorphism
	\begin{equation*}
	\varprojlim_{H'}\bS_{M} \cong \varprojlim_{H}\bS_{M} \times_{\bS_{N'}} \bS_{N}
    \end{equation*}
	which is smooth and irreducible \cite[Proposition~A.2.3]{Corey17} of dimension 10 \cite[Proposition~III.9.5]{Hartshorne}. Next, observe that $M_{6}$ is equivalent to the $\Delta$-matroid $K$ from Example~\ref{ex:SpecialDeltaMatroid}. 	Thus $\varprojlim_{\Gamma_w}\bS_{M}$ fits into the pullback diagram
	\begin{equation*}
	\xymatrix
	{
		{\varprojlim_{\Gamma_w}\bS_{M}}\ar[r]^-{} \ar[d]_-{f} & \bS_{M_6} \ar[d]^{(\varphi_{M_6,M_{06}}, \varphi_{M_6,M_{56}})}\\
		{\varprojlim_{H'}\bS_{M}} \ar[r]^{} & \bS_{M_{06}} \times \bS_{M_{56}} 
	}
	\end{equation*} 
	The dimension of  $\varprojlim_{\Gamma_w}\bS_M$ is at least 10 by Theorem \ref{thm:closedImmersion} and the fact that $\init_w\bS_{5}^{\circ}$ a flat degeneration of $\bS_{5}^{\circ}$, which is 10-dimensional. We conclude that $\varprojlim_{\Gamma_w}\bS_{M}$ is smooth an irreducible of dimension 10 by Lemma \ref{lem:pullbackClosedImmersion}, and therefore so is $\init_{w}\bS_{5}^{\circ}$.

	Finally, we consider subdivision 15. We retain the notation convention for the vertices, polytopes, and $\Delta$-matroids.  The set $A=\{Q_0,\ldots, Q_{8}\}$ is basis-connected, so we may compute its coordinate ring using Proposition \ref{prop:coordinateRingLimit}. We see that $I_{M_i}^{x} = \langle 0 \rangle$ for $i=0,1,2,3,5,7,8$ and $I_{M_i}^{x} = \langle x_{02}x_{13}-x_{03}x_{12}\rangle$ for $i=4,6$. By Proposition \ref{prop:coordinateRingLimit}, the coordinate ring of $\varprojlim_{H}\bS_{M_i}$, where $H = \Gamma_w[v_0,\ldots, v_8]$, is
	\begin{align*}
	R^x(A) \cong S^{-1}\kk[x_{ij}^{\pm} \, : \, 0\leq i<j\leq 4] / (x_{02}x_{13}-x_{03}x_{12})
	\end{align*}
	where $S$ is some finitely-generated multiplicative semigroup. As in the previous case, we deduce that $\varprojlim_{H}\bS_{M_i}$ is isomorphic to an open subvariety of $\GGm^{9}$, and therefore smooth and irreducible of dimension 9. 
	
	Next, consider the pair $(Q_9, Q_{69})$. Twist $M_{9}$ and $M_{69}$ by $1234$ to get the $\Delta$-matroids $N$ and $N'$ with bases
	\begin{align*}
	\cB(N) &= \{1, 2, 3, 012, 013, 123, 134, 234, 01234 \}, \\
	\cB(N') &= \{1, 2, 3, 012, 013, 134, 234, 01234\}.
	\end{align*}
	The rings of $R_{N}^{x}$ and $R_{N'}^{x}$ are 
	\begin{align*}
	R_{N}^{x} = \kk[x_{01}^{\pm}, x_{02}^{\pm}, x_{04}^{\pm}, x_{24}^{\pm}, x_{34}^{\pm}], \hspace{10pt}  R_{N'}^{x} = \kk[x_{01}^{\pm}, x_{02}^{\pm}, x_{24}^{\pm}, x_{34}^{\pm}]
	\end{align*}
	and the map $\varphi_{N,N'}^{\#}: R_{N'}^{x} \to R_{N}^{x}$ is given by $x_{ij}\mapsto x_{ij}$, so $\varphi_{N,N'}:\bS_{N} \to \bS_{N'}$ is smooth and dominant with connected fibers. Set $H' = \Gamma_w[v_0,\ldots, v_9]$.  By \cite[Proposition~A.5]{Corey17}, we have an isomorphism
	\begin{equation*}
	\varprojlim_{H'}\bS_{M} \cong \varprojlim_{H}\bS_{M} \times_{\bS_{N'}} \bS_{N}
	\end{equation*}
	which is smooth and irreducible of dimension 10 (similar to the previous case). Finally, $\Delta$-matroids $M_{10}$ and $M_{11}$ are equivalent to the matroid $K$ from Example \ref{ex:SpecialDeltaMatroid}, and by an argument similar to the previous case, we conclude that $\init_w\bS_{5}^{\circ}$ is smooth and irreducible.
\end{proof}

\section{Log canonical model of $\chow{\bS_{5}}{H}$}
\label{sec:chow}

\subsection{Sch\"on compactifications}
Throughout this section, we use the following notation for fans and toric varieties that is consistent with \cite{Fulton}. Given a finite rank lattice $N$, let $T_{N}$ be its torus, and given a torus  $T$ let $N_T$ be its cocharacter lattice. If $\tau$ is a rational polyhedral cone, let $N_{\tau}$ be the saturated sublattice of $N$ generated by $\tau\cap N$, let $N(\tau) = N/N_{\tau}$, and let $\St(\tau)$ be the star of $\tau$, viewed as a fan in $N(\tau)_{\RR}$. If $\Sigma$ is a rational polyhedral fan in $N_{\RR}$, denote by $|\Sigma|$ its support. If $\Sigma$ is also pointed, denote by $X(\Sigma)$ its toric variety.

Suppose $Y^{\circ}$ is a closed subvariety of an algebraic torus $T$. The closure $Y$ of $Y^{\circ}$ in a $T$-toric variety $X(\Sigma)$ is a \textit{tropical}, resp. \textit{sch\"on}, \textit{compactification} if the multiplication map $Y \times T \to X(\Sigma)$ is flat, resp. smooth, and surjective; in either case $|\Sigma| = \Trop(Y^{\circ})$ \cite{Tevelev}.  The variety $Y^{\circ}$ is \textit{sch\"on} if it admits a sch\"on compactification, equivalently, if $\init_w Y^{\circ}$ is smooth for each $w\in \Trop(Y^{\circ})$ \cite[Proposition~3.9]{HelmKatz}. If $Y^{\circ}$ is sch\"on, then the closure of $Y^{\circ}$ inside any toric variety $X(\Sigma)$ with  $|\Sigma| = \Trop(Y^{\circ})$ is a sch\"on compactification \cite[Theorem~1.5]{LuxtonQu}.

A rational pointed polyhedral fan $\Sigma$ in $N_{\RR}$ is \textit{strictly simplicial} if, for each cone $\tau$ of $\Sigma$, its rays can be extended to an integral basis of $\Sigma$. The following proposition is well known to the experts, but for convenience we sketch a proof here. 

\begin{proposition}
	\label{prop:schonSmoothSNC}
	Suppose $Y^{\circ}$ is sch\"on and $\Sigma$ is a strictly simplicial fan with support $\Trop(Y^{\circ})$. Then the closure $Y$ of $Y^{\circ}$ in $X(\Sigma)$ is smooth and its boundary $B = Y\setminus Y^{\circ}$ is a simple normal crossings divisor. 
\end{proposition}

\begin{proof}
	As stated earlier, $Y$ is a sch\"on compactification of $Y^{\circ}$, so the multiplication map $Y\times T \to X(\Sigma)$ is smooth and surjective. Because $\Sigma$ is strictly simplicial, the toric variety $X(\Sigma)$ is smooth and its toric boundary is a simple normal crossings divisor. So $Y\times T$ is smooth, and therefore so is $Y$. Moreover, since the toric boundary of $X(\Sigma)$ pulls back to $B \times T$ under the smooth and surjective multiplication map, we have that $B \times T$, and therefore $B$, is a simple normal crossings divisor by \cite[\href{https://stacks.math.columbia.edu/tag/0CBQ}{Lemma 0CBQ}]{stacks-project}.
\end{proof}

Let $Y$ be a sch\"on compactification of $Y^{\circ}$ with ambient toric variety $X(\Sigma)$, and set $B=Y\setminus Y^{\circ}$. Then $B$ is divisorial and $Y$ has toroidal singularities \cite[Theorem~1.4]{Tevelev}. Given a cone $\tau$ of $\Sigma$, denote by $Y_{\tau}$ the intersection of $Y$ with the stratum of $X(\Sigma)$ corresponding to $\tau$.  The following proposition gives a sufficient condition for $K_{Y}+B$ to be ample, compare to \cite[Theorem~4.9]{LuxtonQu} and \cite[Theorem~1.3]{Corey17}. 

\begin{proposition}
	\label{prop:KYBAmple}
	 If $\init_{w}Y^{\circ}$ is smooth and irreducible for all $w\in \Trop(Y^{\circ})$, and $|\St(\tau)|$ is not preserved by a rational subspace of $N(\tau)_{\RR}$ for all $\tau \in \Sigma$, then $K_{Y}+B$ is ample (where $B=Y\setminus Y^{\circ}$).
\end{proposition}

\begin{proof}
	Because $Y$ has toroidal singularities, the divisor $K_Y+B$ is ample if and only if each irreducible stratum is log minimal \cite[Theorem~9.1]{HackingKeelTevelev2009}.  The stratum $Y_{\tau}$ is irreducible because the initial degenerations of $Y^{\circ}$ are irreducible, and $\init_{w} Y^{\circ} \cong Y_{\tau} \times T_{N_{\tau}}$ for any $w$ in the relative interior of $\tau$ \cite[Lemma~3.6]{HelmKatz}. Therefore, the irreducible strata of $Y$ are exactly  $Y_{\tau}$ for $\tau \in \Sigma$. 
	
   Given $\tau \in \Sigma$, the stratum $Y_{\tau}$ is sch\"on because $Y$ is a sch\"on compactification. Therefore, $Y_{\tau}$ is log minimal if and only if it is not preserved by a nontrivial subtorus $T_{N(\tau)}$ \cite[Theorem~3.1]{HackingKeelTevelev2009}, which occurs if and only if $\Trop(Y_{\tau})$ is not invariant under translation by a rational subspace of $N(\tau)_{\RR}$ \cite[Lemma~5.2]{KatzPayne2011}.  The proposition now follows from the fact that $\Trop(Y_{\tau}) = |\St(\tau)|$ \cite[Lemma~3.3.6]{MaclaganSturmfels2015}.
\end{proof}

We use the following lemma, which is essentially \cite[Lemma~7.2]{Corey17}, to determine whether $|\St(\tau)|$ is preserved under translation by a subspace.

\begin{lemma}
	\label{lem:notPreserved}
	Suppose $\Sigma$ is a fan with lineality space $L_{\RR}$, and let $\tau \neq L_{\RR}$ be a nonmaximal cone of $\Sigma$. If there is a collection $A_{\tau}$ of maximal cones such that 
	\begin{equation}
	\label{eq:notPreserved}
	\bigcap_{\sigma\in A_{\tau}}  (N_{\sigma})_{\RR} = (N_{\tau})_{\RR}
	\end{equation}
	then $|\St(\tau)|$ is not preserved under translation by any subspace of $N(\tau)_{\RR}$. 
\end{lemma}

\subsection{The Chow quotient of $\bS_n$}
As before, let $N = \ZZ^{E(n)}/\ZZ\!\cdot\!(1,\ldots,1)$ be the cocharacter lattice of the dense torus $T_N$ of $\PP(\kk^{E(n)})$, and let $L\leq N$ be the saturated subgroup from Equation \eqref{eq:Lineality}. This is the cocharacter lattice of an $n$-dimensional subtorus $T_L\cong H$ of $T_N$, and the scaling action of $T_N \curvearrowright \PP(\kk^{E(n)})$ restricted to $\bS_n$ is the action $H\curvearrowright \bS_n$ from Equation \eqref{eq:Haction}. Therefore, there is an embedding of Chow quotients $\chow{\bS_n}{H} \hookrightarrow \chow{\PP(\kk^{E(n)})}{H}$. By \cite{KapranovSturmfelsZelevinsky}, $\chow{\PP(\kk^{E(n)})}{H}$ is the toric variety of the secondary fan of $\Delta(n)$.  Denote by $\Sigma_{n}$ restriction of this fan to $\Dr(n)$, which has $L_{\RR}$ in its lineality space. Thus $\Dr(n)/L_{\RR}$ is the underlying set of the fan $\Sigma_{n}/L_{\RR}$, and because $\TS_{n}^{\circ}\subset \Dr(n)$, the Chow quotient $\chow{\bS_n}{H}$ is the closure of $\bS_n^{\circ}/H$ in $X(\Sigma_{n}/L_{\RR})$. 

\begin{remark}
	Each point of the affine variety $\bS_n^{\circ}$ is $H$-stable in the sense of geometric invariant theory since, for each $x\in \bS_n^{\circ}$, the stabilizer is $\{\pm I_n\}$, which is finite, and the orbit $H\cdot x = (T_{L} \cdot x) \cap \bS_n^{\circ}$ is closed in $\bS_{n}^{\circ}$. Therefore,  $\bS_{n}^{\circ} / H$ is a geometric quotient by standard results in GIT.   
\end{remark}

\begin{lemma}
	\label{lem:chowQuotientSchon}
	For $n\leq 5$, the initial degenerations of $\bS_{n}^{\circ}/H$ are smooth and irreducible; in particular, $\bS_{n}^{\circ}/H$ is sch\"on. 
\end{lemma}

\noindent Note that $\bS_n^{\circ}$ is sch\"on by Theorem \ref{thm:schonSpinor}.

\begin{proof}
	Compare to \cite[Lemma~7.1]{Corey17}. Let $w \in \TS_{n}$ and $\overline{w}$ the projection of $w$ to $N_{\RR}/L_{\RR}$. Then $\init_{w}\bS_{n}^{\circ} \cong H \times \init_{\overline{w}} (\bS_{n}^{\circ}/H)$. Therefore, for $n\leq 5$, the initial degenerations of $\bS_{n}^{\circ}/H$ are smooth and irreducible by Theorem \ref{thm:schonSpinor}.
\end{proof}

\noindent In particular, for $n\leq 5$, the Chow quotient $\chow{\bS_{n}}{H}$ is a sch\"on compactification of $\bS_n^{\circ}/H$. 
When $n=1,2,3$, the spinor variety is $\bS_{n} = \PP(\kk^{E(n)})$ and $\chow{\bS_{n}}{H}$ is a projective toric variety. The first interesting case is $n=4$.

\begin{proposition}
	\label{prop:n4LogCanonical}
	The log canonical divisor $K_{\chow{\bS_4}{H}} + B_4$ is ample, and therefore $\chow{\bS_{4}}{H}$ is log canonical.
\end{proposition}

\begin{proof}
	 The secondary fan of $\TS_{4}^{\circ}$ is described in \S \ref{sec:tropSpinor}. By Proposition \ref{prop:KYBAmple} and Lemma \ref{lem:chowQuotientSchon}, we need to show that $|\St(\tau)|$ is not invariant under translation by any rational subspace of $N(\tau)_{\RR}$. It suffices to consider the nonmaximal cones $\tau\in \Sigma_4$. Up to $W(D_4)$-symmetry, there is only one cone, e.g., $\tau=\RR_{\geq 0}\langle f_{03} + f_{12} \rangle +L_{\RR}$. There are 3 maximal cones containing $\tau$:
	\begin{gather*}
	\sigma_0=\RR_{\geq 0}\langle f_{\emptyset} + f_{0123} , f_{03} + f_{12}  \rangle +L_{\RR}, \hspace{15pt} \sigma_1=\RR_{\geq 0}\langle f_{01} + f_{23} , f_{03} + f_{12}  \rangle +L_{\RR}, \\
	\sigma_2=\RR_{\geq 0}\langle f_{02} + f_{13} , f_{03} + f_{12}  \rangle +L_{\RR}.
	\end{gather*}
	Given $v\in N_{\RR}$, denote by $v^{*}$ the linear functional given by $u\mapsto \bracket{u}{v}$. The subspace $(N_{\sigma_i})_{\RR}$ is the kernel of $v_i^*$ where
	\begin{gather*}
	v_0 = f_{01} - f_{02} - f_{13} + f_{23},\; v_1 = f_{\emptyset} - f_{02} - f_{13} + f_{0123},\; v_2 = f_{\emptyset} - f_{01} - f_{23} + f_{0123}.
	\end{gather*}
	Because $\Span\{v_0,v_1,v_2\}$ is 2-dimensional, the intersection in Equation \eqref{eq:notPreserved} holds for $A_{\tau} = \{\sigma_0,\sigma_1,\sigma_2 \}$, and hence  $|\St(\tau)|$ is not invariant under translation by any rational subspace of $N(\tau)_{\RR}$ by Lemma \ref{lem:notPreserved}.
\end{proof}

The rest of this section is devoted to the $n=5$ case. Recall that $\TS_{5}^{\circ} = \Dr(5)$, and the Gr\"obner fan of $\TS_{5}^{\circ}$ equals $\Sigma_5$.  Here is an explicit description of $\Sigma_5$, adapted from a \texttt{gfan} computation.   By \S \ref{sec:tropSpinor}, the lineality space of $\Sigma_5$ is the 5-dimensional subspace $L_{\RR}$ where $L\subset N$ is the saturated subgroup generated by 
\begin{align*}
\ell_i = \sum_{\lambda \ni i} f_{\lambda}, \hspace{10pt} m_i = \sum_{\lambda \not\ni i} f_{\lambda} \hspace{10pt} \text{ for } \hspace{10pt} 0\leq i \leq 4.
\end{align*} 
This fan has $36$ rays, which have primitive vectors (modulo $L_{\RR}$):
\begin{align*}
\begin{array}{llllll}
r_{0}= f_{0} & r_{1}= f_{1} & r_{2}= f_{2} & r_{3}= f_{3} & r_{4}= f_{4} & r_{5}= f_{012} \\
r_{6}= f_{013} & r_{7}= f_{023} & r_{8}= f_{123} & r_{9}= f_{014} & r_{10}= f_{024} & r_{11}= f_{124} \\
r_{12}= f_{034} & r_{13}= f_{134} & r_{14}= f_{234} & r_{15}= f_{01234} 
\end{array}
\end{align*}

\begin{align*}
\begin{array}{lll}
r_{16} = f_0+f_1+f_2+f_{012} & r_{17} = f_0+f_1+f_3+f_{013} & r_{18} = f_0+f_2+f_3+f_{023} \\ 
r_{19} = f_1+f_2+f_3+f_{123} & r_{20} = f_0+f_1+f_4+f_{014} & r_{21} = f_0+f_2+f_4+f_{024} \\ 
r_{22} = f_1+f_2+f_4+f_{124} & r_{23} = f_0+f_3+f_4+f_{034} & r_{24} = f_1+f_3+f_4+f_{134} \\ 
r_{25} = f_2+f_3+f_4+f_{234}
\end{array}
\end{align*}

\begin{align*}
\begin{array}{lll}
r_{26} = -f_0-f_1-f_2-f_{012} &	r_{27} = -f_0-f_1-f_3-f_{013} &	r_{28} = -f_0-f_2-f_3-f_{023} \\ 
r_{29} = -f_1-f_2-f_3-f_{123} &	r_{30} = -f_0-f_1-f_4-f_{014} &	r_{31} = -f_0-f_2-f_4-f_{024} \\
r_{32} = -f_1-f_2-f_4-f_{124} &	r_{33} = -f_0-f_3-f_4-f_{034} &	r_{34} = -f_1-f_3-f_4-f_{134} \\	
r_{35} = -f_2-f_3-f_4-f_{234}
\end{array}
\end{align*}
Notice that, e.g., $r_{25}$ and $r_{35}$ belong to the same $W(D_5)$-orbit modulo $L$ because
\begin{equation*}
t_{02}\cdot r_{25} - r_{35} = m_{1}
\end{equation*}
The cones of $\Sigma_5$ partition into 20 $W(D_5)$-orbits, which we denote by $\cO_{i}$ for $-1\leq i \leq 18$. Orbit representatives are listed in Table \ref{table:cones}.

The space $\TS_5^{\circ}$ is supported on a slightly coarser fan $\Sigma_5'$, which we now describe. 
\begin{lemma}
	\label{lem:star11}
 Each $\tau\in \cO_{11}$ has exactly two cones in its star, and these belong to the orbit $\cO_{18}$. Similarly, each $\sigma \in \cO_{18}$ has only one face in $\cO_{11}$. 
\end{lemma}
That is, the cones in $\cO_{18}$ partition in pairs, one pair $(\sigma_{\tau}, \sigma_{\tau'})$ for each $\tau\in \cO_{11}$. For example, if $\tau = \RR_{\geq 0} \langle r_{3}, r_{4}, r_{6}, r_{9} \rangle + L_{\RR}$, then 
\begin{equation}
	\label{eq:exampleTauStar}
\sigma_{\tau} = \RR_{\geq 0} \langle r_{3}, r_{4}, r_{6}, r_{9}, r_{25} \rangle + L_{\RR} 
\hspace{20pt}  \sigma_{\tau}' = \RR_{\geq 0} \langle r_{3}, r_{4}, r_{6}, r_{9}, r_{26} \rangle + L_{\RR}.
\end{equation}

\begin{proof}
	For the first part, it suffices to show that the star of $\tau = \RR_{\geq 0} \langle r_{3}, r_{4}, r_{6}, r_{9} \rangle + L_{\RR}$ consists of exactly $\sigma_{\tau}$ and $\sigma_{\tau}'$ from Equation \eqref{eq:exampleTauStar}. That $\St(\tau)$ contains no cones from the orbits $\cO_{15}, \cO_{16}, \cO_{17}$ follows from the fact that, for 
	\begin{gather*}
	S_{15} = \{f_{4}, f_{023}, f_{123}, f_{014}  \}, \hspace{5pt} S_{16} = \{f_{3}, f_{4}, f_{012}, f_{023}, f_{014}\}, \hspace{5pt}S_{17} = \{f_{3}, f_{4}, f_{012}, f_{013}\},
	\end{gather*}
	there is no even subset $\mu$ such that $t_{\mu} \cdot S_m$ contains some $\{f_{k},f_{\ell},f_{ijk}, f_{ij\ell}\}$ for some distinct $i,j,k,\ell$,  see \S\ref{sec:symmetries}. Now, the $W(D_5)$-stabilizer of $S_{18} = \{f_{3},f_{4},f_{013},f_{014}\}$ is 
	\begin{equation*}
	\Stab_{W(D_5)}(S_{18}) = \langle s_{(01)}, s_{(34)}, t_{01}, t_{34}, s_{(03)(14)}t_{03}  \rangle. 
	\end{equation*}
 	The $\Stab_{W(D_5)}(S_{18})$-orbit of $r_{25}= f_{2}+f_{3}+f_{4}+f_{234}$ modulo $L_{\RR}$ is $\{r_{25},r_{26}\}$, and therefore the only cones in the star of $\tau$ are $\sigma_{\tau}$ and $\sigma_{\tau}'$, as required. 
	
	For the second statement, it suffices to consider the cone $\sigma=\RR_{\geq 0} \langle r_{3}, r_{4}, r_{6}, r_{9}, r_{25} \rangle + L_{\RR}$. Every facet of $\sigma$ other than $\RR_{\geq 0} \langle r_{3}, r_{4}, r_{6}, r_{9} \rangle + L_{\RR}$ has one ray in the $W(D_5)$-orbit of $\RR_{\geq 0}r_{25} +L_{\RR}$, and therefore is not in $\cO_{11}$. 
\end{proof}

\begin{lemma}
	\label{lem:glue18}
    For any $\tau\in \cO_{11}$, we have that $\sigma_{\tau} \cup \sigma_{\tau}' $ is the convex polyhedral cone spanned by the rays of $\sigma_{\tau}$ and $\sigma_{\tau}'$. Furthermore, each face of this cone is contained in $\cO_j$ for some $j\neq 11,18$. 
\end{lemma}

\begin{proof}
By symmetry, it suffices to show that  
	\begin{equation*}
	\sigma_{\tau} \cup \sigma_{\tau}' =  \RR_{\geq 0} \langle r_{3}, r_{4}, r_{6}, r_{9}, r_{25}, r_{26} \rangle + L_{\RR}
	\end{equation*}
	where $\tau = \RR_{\geq 0} \langle r_{3}, r_{4}, r_{6}, r_{9} \rangle + L_{\RR}$.
	Because the cones $\sigma_{\tau}/L_{\RR}$ and $\sigma_{\tau}'/L_{\RR}$ are simplicial, we must show that
	\begin{equation*}
	(\RR_{\geq 0} \langle r_{25}, r_{26} \rangle) \cap \tau \neq L_{\RR}.
	\end{equation*}
	Indeed, consider the vector $w = f_{3}+f_{4}+f_{013}+f_{014}$. 	We may express $w$ as
	\begin{equation*}
	w = r_{3}+r_{4}+{r_6}+r_{9} = r_{25} + r_{26} +\tfrac{1}{2}( \ell_{0} + \ell_{1} - \ell_{2} - \ell_{3} + m_{4}).
	\end{equation*}
	This shows that $w$ lies in $(\RR_{\geq 0} \langle r_{25}, r_{26} \rangle) \cap \tau$ but not in $L_{\RR}$, as required. The last statement follows from Lemma \ref{lem:star11}.
\end{proof}

The fan $\Sigma_5'$ is obtained by gluing $\sigma_{\tau}$ and $\sigma_{\tau}'$ along $\tau$, for each $\tau\in\cO_{11}$. Explicitly, let $\cO'_{18} = \set{\sigma_{\tau}\cup \sigma_{\tau}'}{\tau \in \cO_{11}}$.
Define $\Sigma_5'$ by
\begin{equation*}
\Sigma_5' =  \set{\tau}{\tau\in \cO_{i} \text{ for } i\neq 11,18, \text{ or } \tau  \in \cO_{18}'}.
\end{equation*}

\begin{lemma}
	The collection of cones $\Sigma_5'$ is a fan whose support is $\TS_5^{\circ}$, and $\Sigma_5$ is a refinement of $\Sigma_5'$.
\end{lemma}

\begin{proof}
	It is clear that $\TS_5^{\circ}$ is the union of the cones in $\Sigma_5'$	and each cone of $\Sigma_5$ is contained in a cone of $\Sigma_5'$. Therefore, it suffices to show that $\Sigma_5'$ is a fan.  That is, we must show that every face of a cone in $\Sigma_5'$ is in $\Sigma_5'$, and the intersection of any two cones in $\Sigma_5'$ is a face of both of them. 
	
	Suppose $\sigma\in \Sigma_5'$, and $\sigma'$ is a face of $\sigma$. If $\sigma\in \cO_{i}$ for some $i\neq 11,18$, then $\sigma'$ is in $\cO_{j}$ for some $j\neq 11,18$ by Lemma \ref{lem:star11} and because $\Sigma_5$ is a fan; in particular $\sigma' \in \Sigma_5'$. If $\sigma \in \cO_{18}'$ then $\sigma'\in \cO_{j}$ for some $j\neq 11,18$ by Lemma \ref{lem:glue18}.

	Suppose $\sigma_1,\sigma_2\in \Sigma_5'$ and let $\sigma_3=\sigma_1 \cap \sigma_2$. If $\sigma_1\in \cO_i$ and $\sigma_2\in\cO_j$ for some $i,j\neq 11,18$, then $\sigma_{3}\in \cO_k$ for some $k\neq 11,18$ because $\Sigma_5$ is a fan and Lemma \ref{lem:star11}. Now  suppose 
	$\sigma_2\in \cO_{18}'$, by symmetry it suffices to consider $\sigma_2 = \sigma_{\tau} \cup \sigma_{\tau}'$ from Equation \eqref{eq:exampleTauStar}. 
	If $r_{25}, r_{26}\in \sigma_1$, then $\sigma_1$ meets the relative interior of $\sigma_{\tau}$ and $\sigma_{\tau}'$ by Lemma \ref{lem:glue18}, and hence $\sigma_1=\sigma_2=\sigma_3$ because $\Sigma_5$ is a fan. If $r_{25}\in \sigma_1$ but  $r_{26}\notin \sigma_1$, then $\sigma_3$ is a face of $\sigma_{\tau}$ not contained in $\tau$, and therefore is a face of $\sigma_1$ and $\sigma_2$. Finally, if that $r_{25}, r_{26} \notin \sigma_1$, then $\sigma_3$ is a face of $\sigma_1$ and $\tau$, but it cannot equal $\tau$ by Lemma \ref{lem:star11}. Therefore, $\sigma_3$ is also a face of $\sigma_2$, as required.	
\end{proof}

Let $Y$ denote the closure of $\bS_{5}^{\circ}/H$ in  the toric variety $X(\Sigma_5'/L_{\RR})$.

\begin{lemma}
	\label{lem:YInvariantSubspace}
	For each $\tau\in \Sigma_5'/L_{\RR}$, the space $|\St(\tau)|$ is not preserved under translation by any rational subspace of $N(\tau)_{\RR}$.
\end{lemma}

\begin{proof}
	By Lemma \ref{lem:notPreserved},  we must find a collection of maximal cones $A_{\tau}$ satisfying Equation \eqref{eq:notPreserved}.  We need only check one cone, e.g., $\tau_i$, from each non-maximal $W(D_5)$-orbit $\cO_i$.  Sets $A_{\tau_i}$ that satisfy \eqref{eq:notPreserved} are listed in Table~\ref{table:cones}. 
	
	Equation \eqref{eq:notPreserved} may be verified for $A_{\tau}$ in the following way. For each $\sigma\in A_{\tau}$, let $R_{\sigma\tau}$ be the matrix whose rows consist of the rays of $\sigma$ together with the vectors $\ell_0, \ldots, \ell_4, m_0, \ldots, m_4$ spanning the lineality space. Let $S_{\sigma\tau}$ be a matrix such that $\rowsp R_{\sigma\tau} = \ker S_{\sigma\tau}$.  Now let $S_{\tau}$ be the matrix whose (block) rows are the $S_{\sigma\tau}$. The kernel of $S_{\tau}$ is  $\bigcap_{\sigma\in A_{\tau}}  (N_{\sigma})_{\RR}$. Thus, to establish \eqref{eq:notPreserved} for $A_{\tau}$, one must show that   $\rank(S_{\tau}) = \codim_{\RR^{E(5)}}(\tau)$, which is a routine verification. 
\end{proof}

\begin{table}[h!]
	\centering
	\begin{tabular}{|c|c|c|c|}
		\hline
		&Rep. $\tau_i\in \cO_i$ & Dim & $A_{\tau_{i}}$  \\
		\hline
		\hline
		-1 & $\emptyset$ & 5 & $t_{23}\cdot \tau_{15}$,\hspace{5pt} $s_{(13)}t_{03}\cdot \tau_{15}$  \\
		\hline
		0 & $\{r_{26}\}$ & 6 & $s_{(12)}\cdot \tau_{15}$,\hspace{5pt} $s_{(34)}\cdot \tau_{15}$,\hspace{5pt} $s_{(12)(34)}\cdot \tau_{15}$ \\
		\hline
		1 & $\{r_4\}$ & 6 & $s_{(142)}t_{14}\cdot \tau_{15}$,\hspace{5pt}  $s_{(024)}t_{03}\cdot \tau_{15}$ \\
		\hline
		2 & $\{r_{3},r_{4}\}$ & 7 & $s_{(021)}t_{34} \cdot \tau_{16}$,\hspace{5pt} $s_{(02)(34)} \cdot \tau_{16}$  \\
		\hline
		3 & $\{r_{4},r_{5}\} $ & 7 & $s_{(032)(14)}t_{23}\cdot \tau_{15}$,\hspace{5pt} $s_{(02)(13)} \cdot \tau_{15}$  \\
		\hline
		4 & $\{r_{4},r_{26}\}$ & 7 & $\tau_{15}$,\hspace{5pt} $s_{(012)}\cdot \tau_{15}$,\hspace{5pt} $s_{(12)}\cdot \tau_{15}$ \\
		\hline
		5 & $\{r_{4},r_{5},r_{9}\}$ & 8 & $\tau_{16}$, \hspace{5pt} $s_{(243)}t_{23}\cdot \tau_{17} $  \\
		\hline
		6 & $\{r_{4},r_{9},r_{26}\}$ & 8 & $\tau_{15}$,\hspace{5pt} $s_{(03)(142)}t_{02} \cdot \tau_{17}$,\hspace{5pt} $s_{(03)(142)}t_{03} \cdot \tau_{17}$ \\
		\hline
		7 & $\{r_4,r_6,r_{26}\}$ & 8 & $s_{(012)}\cdot \tau_{15}$, \hspace{5pt} $s_{(12)}\cdot \tau_{15}$  \\
		\hline
		8 & $\{r_{3},r_{4},r_{5}\}$ & 8 & $\tau_{16} $, \hspace{5pt} $\tau_{17}$  \\
		\hline
		9 & $\{r_{3},r_{4},r_{26}\}$ & 8 & $s_{(04132)}t_{03}\cdot \tau_{17}$, \hspace{5pt} $ s_{(032)(14)}t_{03}\cdot \tau_{17}$  \\
		\hline
		10 & $\{r_{4}, r_{7}, r_{8}, r_{9}\}$ & 9 & $\tau_{15}$,\hspace{5pt} $s_{(23)}\cdot \tau_{15}$ \\
		\hline
		11 & $\{r_{3}, r_{4}, r_{6}, r_{9}\}$ & 9 & N/A\\
		\hline
		12 & $\{r_{3},r_{4},r_{5},r_{25}\}$ & 9 & $\tau_{17}$, \hspace{5pt} $s_{(34)} \cdot \tau_{17}$   \\
		\hline
		13 & $\{r_{3},r_{4},r_{5},r_{6}\}$ & 9 & $s_{(12)}\cdot \tau_{16}$, \hspace{5pt} $s_{(01)(34)} \cdot \tau_{16}$  \\
		\hline
		14 & $\{r_{3},r_{4},r_{6},r_{26}\}$ & 9 & $s_{(03)(14)}t_{03}\cdot \tau_{17}$, \hspace{5pt} $s_{(0314)}t_{03}\cdot \tau_{17}$ \\
		\hline
		15 & $\{r_{4},r_{7},r_{8},r_{9},r_{26}\}$ & 10 & N/A  \\
		\hline
		16 &  $\{r_{3},r_{4},r_{5},r_{7},r_{9}\}$ & 10 & N/A   \\
		\hline
		17 &  $\{r_{3},r_{4},r_{5},r_{6},r_{25}\}$ & 10 & N/A  \\
		\hline
		18 &  $\{r_{3},r_{4},r_{6},r_{9},r_{26}\}$ & 10 & N/A  \\
		\hline
	\end{tabular}
	\caption{Sets $A_{\tau_i}$ satisfying \eqref{eq:notPreserved}}
	\label{table:cones}
\end{table}

\begin{proof}[Proof of Theorem \ref{thm:logcanonical}]
By Lemma \ref{lem:chowQuotientSchon}, the initial degenerations of $\bS_{5}^{\circ}/H$ are smooth and irreducible, so $\chow{\bS_5}{H}$ and $Y$ are sch\"on compactifications of $\bS_{5}^{\circ}/H$. Each $Y_{\tau}$ for $\tau\in \Sigma_5'$ is not preserved by a rational subspace of $N(\tau)_{\RR}$ by Lemma \ref{lem:YInvariantSubspace}. Therefore, $K_{Y}+B$ is ample by Proposition \ref{prop:KYBAmple}, so $Y$ is the log canonical model of $\chow{\bS_{5}}{H}$.

Now consider the last statement. One readily verifies that the fan $\Sigma_5/L_{\RR}$ is strictly simplicial as defined in the beginning of this section. So $\chow{\bS_5}{H}$ is smooth with simple normal crossings boundary by Proposition \ref{prop:schonSmoothSNC}.   The morphism $\chow{\bS_5}{H} \to Y$ is log crepant by \cite[Theorem~1.4]{Tevelev}.
\end{proof}

\subsection*{Funding}
This research was partially supported by NSF RTG Award DMS–1502553 and  ``Symbolic Tools in Mathematics and their Application'' (TRR 195, project-ID 286237555).

\subsection*{Acknowledgments}
The author would like to thank Michael Joswig and anonymous referees for their constructive feedback on earlier drafts.


\pagebreak

\appendix

\section{Subdivisions for $n=5$}
\label{appendix:subdivisionsn5}

In this appendix, we record the matroidal subdivisions of $\Delta(5)$. Each table consists of the subdivisions corresponding to the cones of $\Dr(5)$ of the prescribed dimension. The leftmost column corresponds to the leftmost column of Table \ref{table:cones}. The \textit{Adjacency graph} column records the adjacency graph of the subdivision.  Finally, in the \textit{Bases of the $\Delta$-matroids $M_i$} column, we list the bases of the matroids $M_i$, where $M_i$ is the matroid corresponding to the vertex $i$ in the adjacency graph.

\begin{table}[h!]
	\centering\renewcommand\cellalign{lc}
	\begin{tabular}{|c|c|m{10cm}|}
		\hline
		Index & Adjacency graph & Bases of the $\Delta$-matroids $M_i$ \\
		\hline
		\hline
		0 & \raisebox{-\totalheight+13mm}{\includegraphics[width=4cm]{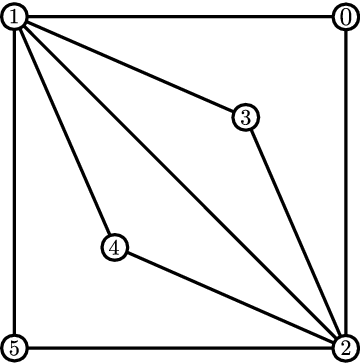}}
		&   \footnotesize{\makecell{$M_{0}: \{1,2,012,123,124,134,234,01234\}$,  \\
				$M_{1}: \{0,1,2,3,012,013,023,123,034,134,234,01234\}$, \\ 
				$M_{2}: \{0,1,2,4,012,014,024,124,034,134,234,01234\}$, \\ 
				$M_{3}: \{0,2,012,023,024,034,234,01234\}$,\\ 
				$M_{4}: \{0,1,012,013,014,034,134,01234\}$, \\ 
				$M_{5}: \{0,1,2,3,4,034,134,234\}$}}  \\
		\hline
		1 &  \raisebox{-\totalheight+3mm}{\includegraphics[width=2.75cm]{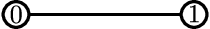}}
		& \footnotesize{\makecell{$M_{0}:\{0,1,2,3,012,013,023,123,014,024,124,034,134,234,01234\}$,\\  
				$M_{1}:\{0,1,2,3,4,014,024,124,034,134,234 \}$	}}   \\
		\hline
	\end{tabular}
	\caption{Dimension 6}
	\label{tab:dim1}
\end{table}

\begin{table}[h!]
	\centering\renewcommand\cellalign{lc}
	\begin{tabular}{|c|c|m{10cm}|}
		\hline
		Index & Adjacency graph & Bases of the $\Delta$-matroids $M_i$  \\
		\hline
		\hline
		2 & \raisebox{-\totalheight+10mm}{\includegraphics[width=3cm]{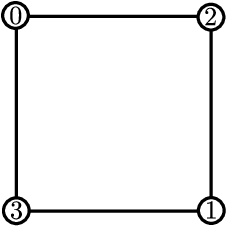}}
		& \footnotesize{\makecell{$M_{0}:\{0,1,2,012,013,023,123,014,024,124,034,134,234,01234 \}$, \\ 
				$M_{1}:\{0,1,2,3,4,034,134,234 \}$, \\ 
				$M_{2}:\{0,1,2,4,014,024,124,034,134,234\}$, \\
				$M_{3}:\{0,1,2,3,013,023,123,034,134,234 \}$
		}}  \\ 
		\hline
		3 & \raisebox{-\totalheight+6mm}{\includegraphics[width=4cm]{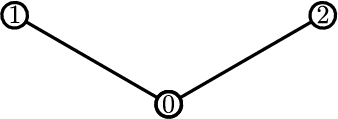}}
		& \footnotesize{\makecell{$M_{0}:\{0,1,2,3,013,023,123,014,024,124,034,134,234,01234\}$, \\ 
				$M_{1}:\{0,1,2,012,013,023,123,014,024,124,01234\}$,\\ 
				$M_{2}:\{0,1,2,3,4,014,024,124,034,134,234\}$ }} \\
		\hline
		4 & \raisebox{-\totalheight+11mm}{\includegraphics[width=4cm]{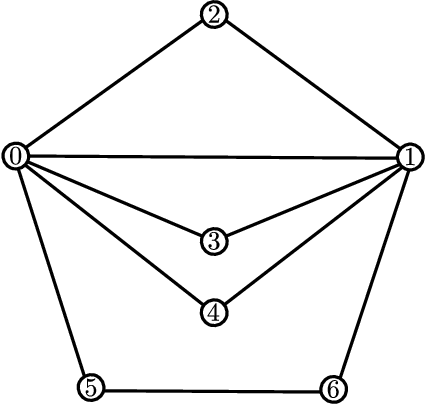}}
		& \footnotesize{\makecell{$M_{0}:\{0,1,2,3,012,013,023,123,034,134,234,01234\}$, \\ 
				$M_{1}:\{0,1,2,012,014,024,124,034,134,234,01234\}$, \\ 
				$M_{2}:\{0,2,012,023,024,034,234,01234\}$, \\ 
				$M_{3}:\{0,1,012,013,014,034,134,01234\}$, \\ 
				$M_{4}:\{1,2,012,123,124,134,234,01234\}$, \\ 
				$M_{5}:\{0,1,2,3,4,034,134,234\}$, \\ 
				$M_{6}:\{0,1,2,3,4,034,134,234\}$}} \\ 
		\hline
	\end{tabular}
	\caption{Dimension 7}
	\label{tab:dim2}
\end{table}

\newpage

\begin{table}[h!]
	\centering\renewcommand\cellalign{lc}
	\begin{tabular}{|c|l|m{9cm}|}
		\hline
		Index & Adjacency graph & Bases of the $\Delta$-matroids $M_i$  \\
		\hline
		\hline
		5& \raisebox{-\totalheight+8mm}{\includegraphics[width=4cm]{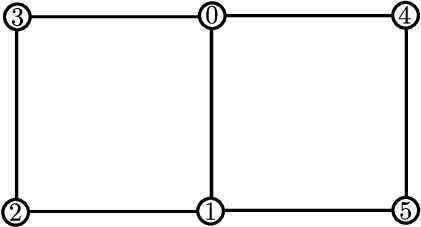}}
		& \footnotesize{\makecell{
				$M_{0}:\{0,1,2,3,013,023,123,024,124,034,134,234,01234\}$, \\  
				$M_{1}:\{0,1,013,014,024,124,034,134,01234\}$, \\ 
				$M_{2}:\{0,1,012,013,014,024,124,01234\}$, \\ 
				$M_{3}:\{0,1,2,012,013,023,123,024,124,01234\}$, \\ 
				$M_{4}:\{0,1,2,3,4,024,124,034,134,234\}$, \\ 
				$M_{5}:\{0,1,4,014,024,124,034,134\}$ }}\\
		\hline
		6& \raisebox{-\totalheight+10mm}{\includegraphics[width=4cm]{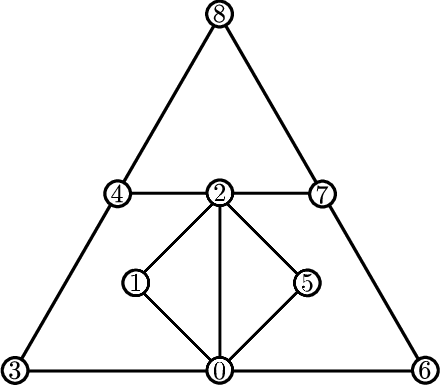}}
		& \footnotesize{\makecell{
				$M_{0}:\{0,1,2,3,012,013,023,123,034,134,234,01234\}$, \\ 
				$M_{1}:\{0,2,012,023,024,034,234,01234\}$, \\ 
				$M_{2}:\{0,1,2,012,024,124,034,134,234,01234\}$, \\ 
				$M_{3}:\{0,1,012,013,014,034,134,01234\}$, \\ 
				$M_{4}:\{0,1,012,014,024,124,034,134,01234\}$, \\
				$M_{5}:\{1,2,012,123,124,134,234,01234\}$, \\ 
				$M_{6}:\{0,1,2,3,4,034,134,234\}$, \\ 
				$M_{7}:\{0,1,2,4,024,124,034,134,234\}$, \\ 
				$M_{8}:\{0,1,4,014,024,124,034,134\}$}} \\ 
		\hline
		7& \raisebox{-\totalheight+12mm}{\includegraphics[width=4cm]{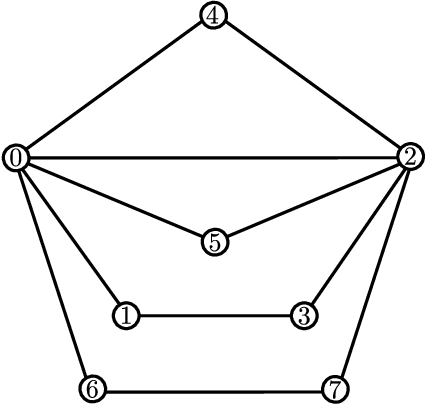}}
		& \footnotesize{\makecell{
				$M_{0}:\{0,1,2,3,012,023,123,034,134,234,01234\}$, \\ 
				$M_{1}:\{0, 1, 3, 012, 013, 023, 123, 034, 134, 01234\}$,\\ 
				$M_{2}:\{0, 1, 2, 012, 014, 024, 124, 034, 134, 234, 01234\}$, \\ 
				$M_{3}:\{0, 1, 012, 013, 014, 034, 134, 01234\}$, \\ 
				$M_{4}:\{1, 2, 012, 123, 124, 134, 234, 01234\}$, \\ 
				$M_{5}:\{0, 2, 012, 023, 024, 034, 234, 01234\}$, \\ 
				$M_{6}:\{0, 1, 2, 3, 4, 034, 134, 234\}$, \\ 
				$M_{7}:\{0, 1, 2, 4, 014, 024, 124, 034, 134, 234\}$}} \\
		\hline
		8& \raisebox{-\totalheight+15mm}{\includegraphics[width=4cm]{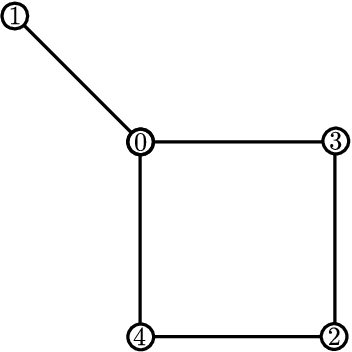}}
		& \footnotesize{\makecell{$M_{0}:\{0, 1, 2, 013, 023, 123, 014, 024, 124, 034, 134, 234, 01234\}$, \\ 
				$M_{1}:\{0, 1, 2, 012, 013, 023, 123, 014, 024, 124, 01234\}$, \\ 
				$M_{2}:\{0, 1, 2, 3, 4, 034, 134, 234\}$, \\ 
				$M_{3}:\{0, 1, 2, 4, 014, 024, 124, 034, 134, 234\}$, \\ 
				$M_{4}:\{0, 1, 2, 3, 013, 023, 123, 034, 134, 234\}$}} \\ 
		\hline
		9& \raisebox{-\totalheight+15mm}{\includegraphics[width=4cm]{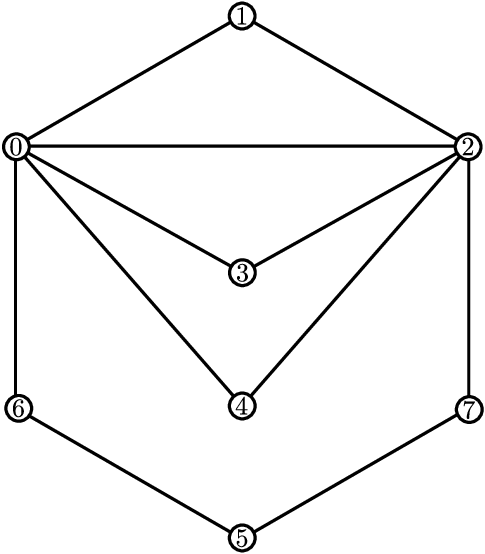}}	
		&  \footnotesize{\makecell{
				$M_{0}:\{0, 1, 2, 012, 014, 024, 124, 034, 134, 234, 01234\}$,\\ 
				$M_{1}:\{0, 2, 012, 023, 024, 034, 234, 01234\}$, \\ 
				$M_{2}:\{0, 1, 2, 012, 013, 023, 123, 034, 134, 234, 01234\}$, \\ 
				$M_{3}:\{1, 2, 012, 123, 124, 134, 234, 01234\}$, \\ 
				$M_{4}:\{0, 1, 012, 013, 014, 034, 134, 01234\}$, \\ 
				$M_{5}:\{0, 1, 2, 3, 4, 034, 134, 234\}$, \\ 
				$M_{6}:\{0, 1, 2, 4, 014, 024, 124, 034, 134, 234\}$, \\ 
				$M_{7}:\{0, 1, 2, 3, 013, 023, 123, 034, 134, 234\}$}}\\
		\hline
	\end{tabular}
	\caption{Dimension 8}
	\label{tab:dim3}
\end{table}

\newpage

\begin{table}[tbh!]
	\centering\renewcommand\cellalign{lc}
	\begin{tabular}{|c|l|m{9cm}|}
		\hline
		Index & Adjacency graph & Bases of the $\Delta$-matroids $M_i$  \\
		\hline
		\hline 
		10& \raisebox{-\totalheight+5mm}{\includegraphics[width=4cm]{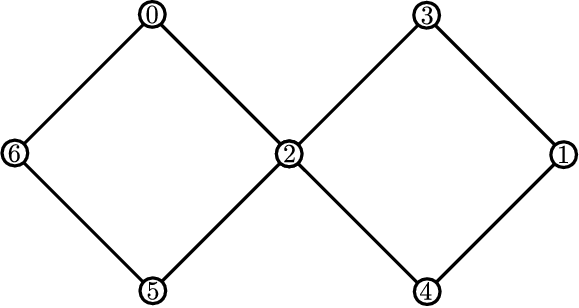}}	
		& \footnotesize{\makecell{
				$M_{0}:\{0, 1, 012, 013, 014, 024, 124, 034, 134, 01234\}$, \\
				$M_{1}:\{2, 3, 012, 013, 023, 123, 234, 01234\}$, \\
				$M_{2}:\{0, 1, 2, 3, 012, 013, 024, 124, 034, 134, 234, 01234\}$, \\
				$M_{3}:\{0, 2, 3, 012, 013, 023, 024, 034, 234, 01234\}$, \\
				$M_{4}:\{1, 2, 3, 012, 013, 123, 124, 134, 234, 01234\}$, \\
				$M_{5}:\{0, 1, 2, 3, 4, 024, 124, 034, 134, 234\}$, \\
				$M_{6}:\{0, 1, 4, 014, 024, 124, 034, 134\}$ }}\\
		\hline
		11& \raisebox{-\totalheight+11mm}{\includegraphics[width=4cm]{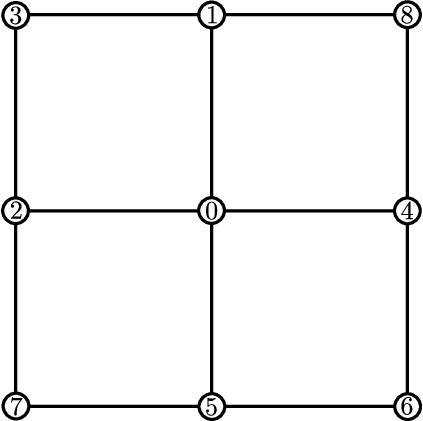}}	
		& \footnotesize{\makecell{
				$M_{0}:\{0, 1, 2, 012, 023, 123, 024, 124, 034, 134, 234, 01234\}$, \\ 
				$M_{1}:\{0, 1, 012, 013, 023, 123, 034,	134, 01234\}$, \\ 
				$M_{2}:\{0, 1, 012, 014, 024, 124, 034, 134, 01234\}$, \\ 
				$M_{3}:\{0, 1, 012, 013, 014, 034, 134,	01234\}$, \\ 
				$M_{4}:\{0, 1, 2, 3, 4, 034, 134, 234\}$, \\ 
				$M_{5}:\{0, 1, 2, 4, 024, 124, 034, 134, 234\}$, \\
				$M_{6}:\{0, 1, 2, 3, 023, 123, 034, 134, 234\}$, \\ 
				$M_{7}:\{0, 1, 4, 014, 024, 124, 034, 134\}$, \\ 
				$M_{8}:\{0, 1, 3, 013, 023, 123, 034, 134\}$}} \\
		\hline
		12& \raisebox{-\totalheight+11mm}{\includegraphics[width=4cm]{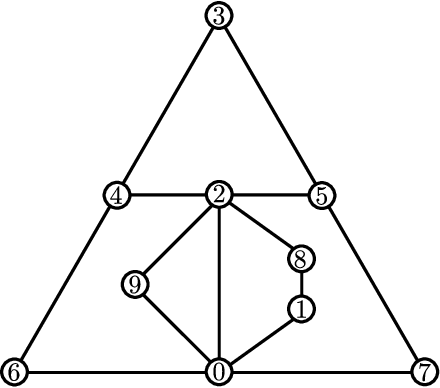}}	
		& \footnotesize{\makecell{
				$M_{0}:\{0, 1, 013, 023, 123, 014, 024, 124, 034, 134, 01234\}$, \\
				$M_{1}:\{0, 1, 012, 013, 023, 123, 014, 024, 124, 01234\}$, \\
				$M_{2}:\{0, 1, 2, 023, 123, 024, 124, 034, 134, 234\}$, \\
				$M_{3}:\{0, 1, 2, 3, 4, 034, 134, 234\}$, \\
				$M_{4}:\{0, 1, 2, 4, 024, 124, 034, 134, 234\}$, \\
				$M_{5}:\{0, 1, 2, 3, 023, 123, 034, 134, 234\}$, \\
				$M_{6}:\{0, 1, 4, 014, 024, 124, 034, 134\}$, \\
				$M_{7}:\{0, 1, 3, 013, 023, 123, 034, 134\}$, \\
				$M_{8}:\{0, 1, 2, 012, 023, 123, 024, 124\}$, \\
				$M_{9}:\{023, 123, 024, 124, 034, 134, 234, 01234\}$}} \\
		\hline
		13 & \raisebox{-\totalheight+11mm}{\includegraphics[width=4cm]{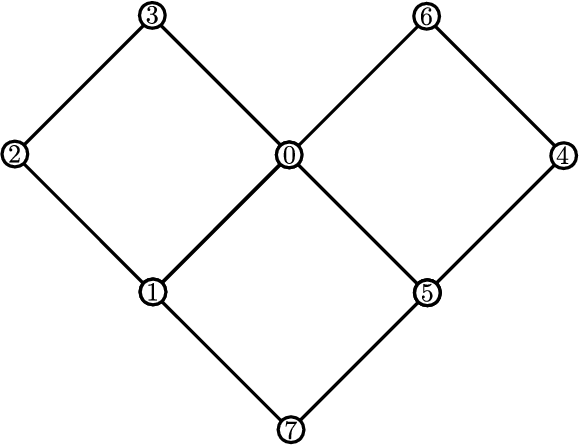}}	
		& \footnotesize{\makecell{
				$M_{0}:\{0, 1, 2, 023, 123, 014, 024, 124, 034, 134, 234, 01234\}$, \\
				$M_{1}:\{0, 1, 013, 023, 123, 014, 034, 134, 01234\}$, \\
				$M_{2}:\{0, 1, 012, 013, 023, 123, 014, 01234\}$, \\
				$M_{3}:\{0, 1, 2, 012, 023, 123, 014, 024, 124, 01234\}$, \\
				$M_{4}:\{0, 1, 2, 3, 4, 034, 134, 234\}$, \\
				$M_{5}:\{0, 1, 2, 3, 023, 123, 034, 134, 234\}$, \\
				$M_{6}:\{0, 1, 2, 4, 014, 024, 124, 034, 134, 234\}$, \\
				$M_{7}:\{0, 1, 3, 013, 023, 123, 034, 134\}$}} \\
		\hline
		14&
		\raisebox{-\totalheight+11mm}{\includegraphics[width=4cm]{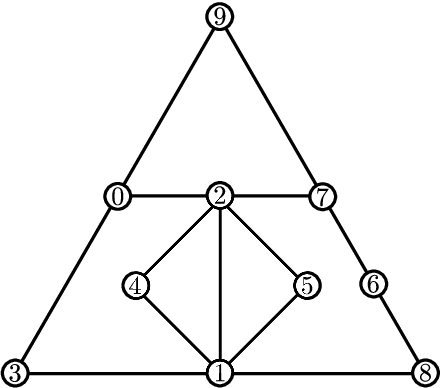}}	
		& \footnotesize{\makecell{
				$M_{0}:\{0, 1, 012, 013, 023, 123, 034, 134, 01234\}$, \\
				$M_{1}:\{0, 1, 2, 012, 014, 024, 124, 034, 134, 234, 01234\}$, \\
				$M_{2}:\{0, 1, 2, 012, 023, 123, 034, 134, 234, 01234\}$, \\
				$M_{3}:\{0, 1, 012, 013, 014, 034, 134, 01234\}$, \\
				$M_{4}:\{1, 2, 012, 123, 124, 134, 234, 01234\}$, \\
				$M_{5}:\{0, 2, 012, 023, 024, 034, 234, 01234\}$, \\
				$M_{6}:\{0, 1, 2, 3, 4, 034, 134, 234\}$, \\
				$M_{7}:\{0, 1, 2, 3, 023, 123, 034, 134, 234\}$, \\
				$M_{8}:\{0, 1, 2, 4, 014, 024, 124, 034, 134, 234\}$, \\
				$M_{9}:\{0, 1, 3, 013, 023, 123, 034, 134\}$}} \\
		\hline
	\end{tabular}
	\caption{Dimension 9}
	\label{tab:dim4}
\end{table}

\newpage

\begin{table}[tbh!]
	\centering\renewcommand\cellalign{lc}
	\begin{tabular}{|c|c|m{9cm}|}
		\hline
		Index & Adjacency graph & Bases of the $\Delta$-matroids $M_i$  \\
		\hline
		\hline 
		15& \raisebox{-\totalheight+12mm}{\includegraphics[width=4cm]{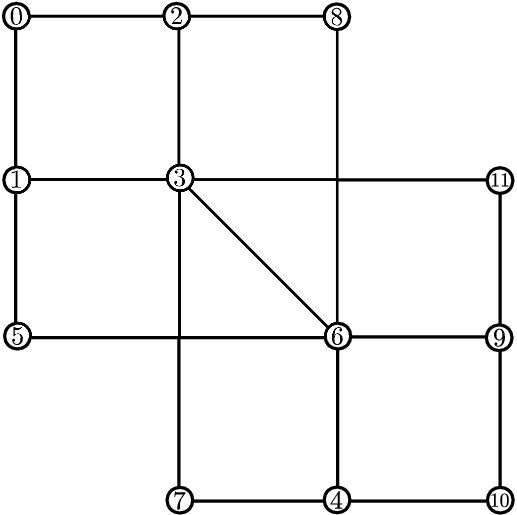}}	
		& \footnotesize{\makecell{
				$M_{0}:\{2, 3, 012, 013, 023, 123, 234, 01234\}$, \\
				$M_{1}:\{0, 2, 3, 012, 013, 023, 034, 234, 01234\}$, \\
				$M_{2}:\{1, 2, 3, 012, 013, 123, 134, 234, 01234\}$, \\
				$M_{3}:\{0, 1, 2, 3, 012, 013, 034, 134, 234, 01234\}$, \\
				$M_{4}:\{0, 1, 012, 014, 024, 124, 034, 134, 01234\}$, \\
				$M_{5}:\{0, 2, 012, 023, 024, 034, 234, 01234\}$, \\
				$M_{6}:\{0, 1, 2, 012, 024, 124, 034, 134, 234, 01234\}$, \\
				$M_{7}:\{0, 1, 012, 013, 014, 034, 134, 01234\}$, \\
				$M_{8}:\{1, 2, 012, 123, 124, 134, 234, 01234\}$, \\
				$M_{9}:\{0, 1, 2, 4, 024, 124, 034, 134, 234\}$, \\
				$M_{10}:\{0, 1, 4, 014, 024, 124, 034, 134\}$, \\
				$M_{11}:\{0, 1, 2, 3, 4, 034, 134, 234\}$}} \\
		\hline
		16& \raisebox{-\totalheight+12mm}{\includegraphics[width=4cm]{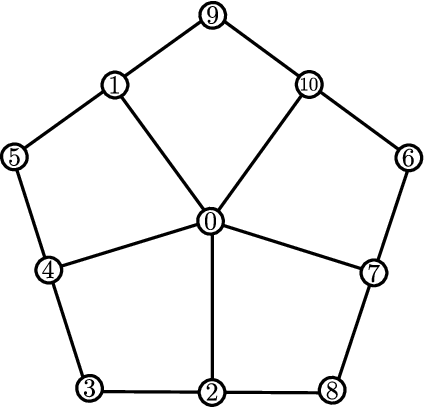}}	
		& \footnotesize{\makecell{
				$M_{0}:\{0, 1, 2, 013, 123, 024, 124, 034, 134, 234, 01234\}$, \\
				$M_{1}:\{0, 2, 013, 023, 123, 024, 034, 234, 01234\}$, \\
				$M_{2}:\{0, 1, 013, 014, 024, 124, 034, 134, 01234\}$, \\
				$M_{3}:\{0, 1, 012, 013, 014, 024, 124, 01234\}$, \\
				$M_{4}:\{0, 1, 2, 012, 013, 123, 024, 124, 01234\}$, \\
				$M_{5}:\{0, 2, 012, 013, 023, 123, 024, 01234\}$, \\
				$M_{6}:\{0, 1, 2, 3, 4, 034, 134, 234\}$, \\
				$M_{7}:\{0, 1, 2, 4, 024, 124, 034, 134, 234\}$, \\
				$M_{8}:\{0, 1, 4, 014, 024, 124, 034, 134\}$, \\
				$M_{9}:\{0, 2, 3, 013, 023, 123, 034, 234\}$, \\
				$M_{10}:\{0, 1, 2, 3, 013, 123, 034, 134, 234\}$}} \\
		\hline
		17& \raisebox{-\totalheight+12mm}{\includegraphics[width=4cm]{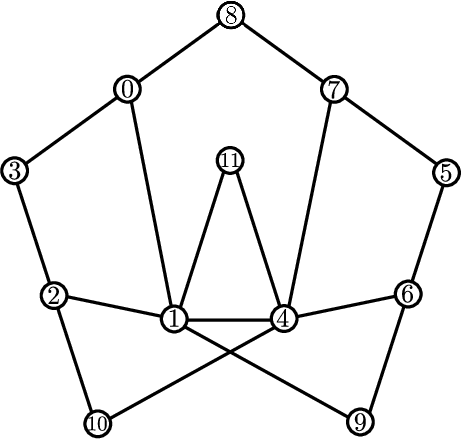}}	
		& \footnotesize{\makecell{
				$M_{0}:\{0, 1, 013, 023, 123, 014, 034, 134, 01234\}$, \\
				$M_{1}:\{0, 1, 023, 123, 014, 024, 124, 034, 134, 01234\}$, \\
				$M_{2}:\{0, 1, 012, 023, 123, 014, 024, 124, 01234\}$, \\
				$M_{3}:\{0, 1, 012, 013, 023, 123, 014, 01234\}$, \\
				$M_{4}:\{0, 1, 2, 023, 123, 024, 124, 034, 134, 234\}$, \\
				$M_{5}:\{0, 1, 2, 3, 4, 034, 134, 234\}$, \\
				$M_{6}:\{0, 1, 2, 4, 024, 124, 034, 134, 234\}$, \\
				$M_{7}:\{0, 1, 2, 3, 023, 123, 034, 134, 234\}$, \\
				$M_{8}:\{0, 1, 3, 013, 023, 123, 034, 134\}$, \\
				$M_{9}:\{0, 1, 4, 014, 024, 124, 034, 134\}$, \\
				$M_{10}:\{0, 1, 2, 012, 023, 123, 024, 124\}$, \\
				$M_{11}:\{023, 123, 024, 124, 034, 134, 234, 01234\}$}} \\
		\hline
		18 & 		\raisebox{-\totalheight+12mm}{\includegraphics[width=4cm]{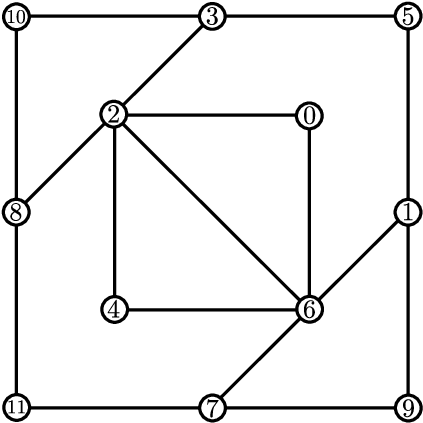}}
		& \footnotesize{\makecell{
				$M_{0}:\{0, 2, 012, 023, 024, 034, 234, 01234\}$, \\
				$M_{1}:\{0, 1, 012, 014, 024, 124, 034, 134, 01234\}$, \\
				$M_{2}:\{0, 1, 2, 012, 023, 123, 034, 134, 234, 01234\}$, \\
				$M_{3}:\{0, 1, 012, 013, 023, 123, 034, 134, 01234\}$, \\
				$M_{4}:\{1, 2, 012, 123, 124, 134, 234, 01234\}$, \\
				$M_{5}:\{0, 1, 012, 013, 014, 034, 134, 01234\}$, \\
				$M_{6}:\{0, 1, 2, 012, 024, 124, 034, 134, 234, 01234\}$, \\
				$M_{7}:\{0, 1, 2, 4, 024, 124, 034, 134, 234\}$, \\
				$M_{8}:\{0, 1, 2, 3, 023, 123, 034, 134, 234\}$, \\
				$M_{9}:\{0, 1, 4, 014, 024, 124, 034, 134\}$, \\
				$M_{10}:\{0, 1, 3, 013, 023, 123, 034, 134\}$, \\
				$M_{11}:\{0, 1, 2, 3, 4, 034, 134, 234\}$}} \\
		\hline
	\end{tabular}
	\caption{Dimension 10}
	\label{tab:dim5}
\end{table}

\end{document}